\documentclass[a4paper,11pt]{article}

\usepackage[utf8]{inputenc}

\usepackage{amsthm}
\usepackage{amsmath}
\usepackage{amssymb}
\usepackage{amsfonts}
\usepackage{dsfont} 
\usepackage{hyperref}
\usepackage{upgreek}
\usepackage{comment}
\usepackage{xfrac}
\usepackage{tikz}
\usepackage{subfig}
\usepackage{appendix}

\usepackage{color}
\usepackage{float}

\usepackage{bm}
\usepackage{bbm}
\theoremstyle{plain}
\newtheorem{theorem}{Theorem}[section]

\theoremstyle{definition}
\newtheorem{Def}{Definition}[section]

\theoremstyle{plain}
\newtheorem{prop}{Proposition}[section]

\theoremstyle{plain}
\newtheorem{lemma}{Lemma}[section]

\theoremstyle{plain}

\theoremstyle{remark}
\newtheorem{remark}{{\it Remark}}

\theoremstyle{remark}

\newcommand{\R}{\mathbb{R}}
\newcommand{\N}{\mathbb{N}}

\newcommand{\nc}{h}

\newcommand{\bd}[1]{\begin{Def}\label{#1}}
\newcommand{\ed}{\end{Def}}
\newcommand{\bp}[1]{\begin{prop}\label{#1}}
\newcommand{\ep}{\end{prop}}
\newcommand{\bt}[1]{\begin{theorem}\label{#1}}
\newcommand{\et}{\end{theorem}}

\def\usigma{ \bm{\upvarsigma}}

\def\usigma{ \bm{\upvarsigma}}
\def\ugamma{ \bm{\upgamma}}
\def\ueta{ \bm{\upeta}}
\newcommand{\eex}{\end{example}}
\newcommand{\de}{\mathrm{d}}

\usepackage{fancyhdr}
\title{Delay-induced periodic behavior in competitive populations}
\author{Michele Aleandri\footnote{Luiss University, Viale Romania, 32, 00197 Rome, Italy, \url{maleandri@luiss.it}},\quad Ida G.~Minelli\footnote{Dipartimento di Ingegneria e Scienze
        dell'Informazione e Matematica, Universit\`a degli Studi
        dell'Aquila, Via Vetoio (Coppito 1), 67100 L'Aquila, Italy,
        \url{idagermana.minelli@univaq.it}}}  
\date{}

\begin{document}
\maketitle
\abstract{We study a model of binary decisions in a fully connected network of interacting agents. Individual decisions are determined by social influence, coming from direct interactions with neighbours, and a group level pressure that accounts for social environment. 
In a competitive environment, the interplay of these two aspects results in the presence of a persistent disordered phase where no majority is formed.   
We sow how the introduction of a delay mechanism in the agent's detection of the global average choice may drastically change this scenario, giving rise to a coordinated self sustained periodic behaviour.  
}


\section{Introduction}
The study of dynamics of social systems through the language and tools of Statistical Physics has  generated growing interest in recent decades in various scientific communities, such as social scientists,  mathematicians, physicists and computer scientists. A social system can in fact be represented  as a multitude of individuals who can randomly change their state by interacting with each other or being influenced by external constraints and inputs.  
Despite their simplicity when compared to the complexity of human social interactions, Statistical Physics models are able to capture several aspects which are often observed in social communities, such as long-term correlations, scaling laws, or the passage, depending on the values of the parameters involved, from a disordered phase in which agents' decisions are unpredictable,
 to a phase in which they coordinate and generate a collective self-organized behaviour: in particular,  phenomena such as synchronization or periodic motions may appear (see, e.g., \cite{castellano2009statistical} \cite{WeidlichHaagSociology},  \cite{schweitzer2003brownian},   \cite{crimaldi2019synchro}, \cite{giacomin2015noise}). This feature is a central issue in the study of social dynamics, whose aim is to understand how the structure of social networks, the nature of interactions and the diverse social responses affect the macroscopic behaviour of the system. 
In this paper we consider a non linear voter model with mean field interaction,  which represents a population of  ''contrarians'' who are subject to a positive social influence exerted by their neighbours. 
More precisely, the agents operate in a competitive environment, so that they are interested in choosing actions that go in the opposite direction to that of the majority, but, at the same time, their behaviour is influenced by interactions with their neighbours and a sort of flock effect pushes them to imitate the agents with whom they have a direct interaction. 
As an example linked to current events, we may think to epidemiological models, where susceptible individuals may face dichotomic behavioural choices, so that,
in order to avoid getting infected, they try to reduce contacts  by preferring choices contrary to that of the majority. At the same time, they may adopt irrational behaviours due to social influence or to an imitation mechanism, adapting their choices to those of their neighbours \cite{covid19} \cite{verelst2016behav}.\\
A similar situation may occur in the context of opinion dynamics or behavioral economics (e.g., \cite{aleandri2019opinion}, \cite{touboul2019hipster}, \cite{grabisch2019model}, \cite{galam2004contrarian}, \cite{galambook}). 
We shall compare this model to the one where agents are cooperative and a fast convergence to consensus occurs. 
Also, we are interested in how perturbations of different types may change the large scale picture of the system as the number of agents grows to infinity, giving rise to periodic self-organized behaviours of the agents.  
This type of phenomenon has already been described in many models of spin systems with mean-field interaction where some kind of frustration is present in the agents' attitude; 
for example, in \cite{regoli13}, \cite{DAIPRA2018},  \cite{AndreisTovazzi},  a dissipation term (that, in absence of interaction, leads the system to a neutral condition where no action is preferred over the other)  
is added in the evolution of Curie-Weiss-like models.\\ 
 In  \cite{collet2016rhythmic} and \cite{touboul2019hipster}  the authors show that periodicity may be triggered when the information about the prevailing choices in the system reaches each agent with a certain amount of delay.  
Notice that the delay hypothesis, besides  being interesting  from a mathematical point of view, is quite natural in a context of social interactions and makes the models more realistic, since usually a single agent does not know in real time what the overall state of the system is.\\  
In particular, in \cite{collet2016rhythmic}, an analysis of the role of delay in the emergence of large scale periodic behaviors for an Ising type two-population model has been done.
Following their approach, we will modify our model by introducing a delay mechanism in the system's evolution and we will show that for the modified 
model a phase transition occurs: depending on the parameters that characterize  the delay,  the mascoscopic system undergoes a Hopf bifurcation and a 
stable limit cycle appears in the dynamics. \\
Going into more detail, we consider a continuous time Markov process representing a population of $N$ agents who may  assume two possible states, 0 or 1, representing two possible actions (or opinions). Allowed transitions are the ones where agents change their state one at a time. For $h\in \{0,1\}$, 
the rate at which each agent switches from state $1-h$ to state $h$ is given by $m_h\phi(m_h)$, where $m_h$ is the fraction of agents that are currently in state $h$ in the population. The function $\phi$ is assumed to be positive and it can be interpreted as the effect of a group pressure, i.e., it  measures the degree to which agents conform or oppose to majority. Indeed, if $\phi$ is increasing, we get  a cooperative (or conformist) population, while if $\phi$ is decreasing, members of the population are competitive (or contrarians). 
Our focus is on the competitive case, for which, while the microscopic system converges a.s. to ''consensus'' (i.e., to one of the two absorbing states, the ''all 0's'' and ''all 1's'' configurations), a stable equilibrium point, representing a mixed phase where both actions coexist in the community, appears in the large scale dynamics as $N\to \infty$. 
Actually, we will show that, with high probability, the time spent by the microscopic system close to such mixed phase is at least exponential in the system's size (from which it follows that the mean absorption time is at least exponential too). \\
We then modify the model and add a delay term in the dynamics by replacing $\phi(m_h)$ in the transition rate defined above with $\phi\Big(\int_0^t K(t-s)m_h(s)\ ds \Big)$, where $m_h(s)$ is 
the fraction of agents that are in state $h$ at time $s$ and $K(t)$ is a  
Gamma distribution delay kernel (also known as \emph{Erlang} kernel), i.e., 
$K(t)=\frac{k^{n+1}}{n!}t^n e^{-kt}$ with $ k>0$ and $n\in\N$. This specific choice of the kernel, besides being effective for application purposes, makes the model more tractable from a mathematical point of view.  For this model in the case when $\phi$ is a strictly decreasing function, using the coefficient $k$ as a bifurcation parameter, we show that, for large $N$, we may still have a  stable equilibrium corresponding to a disordered phase, but  
there exists a critical value of $k$ such that the system loses its stability giving rise to macroscopic self sustained oscillations. 
\section{A model of social interactions under peer-pressure} \label{section2}
\subsection{Definition, macroscopic dynamics and stability }
We consider a family of $N$ agents indexed with an integer $i$ from $1$ to $N$; they can choose two different actions (or  two different opinions) in the set $\{0,1\}$. We denote by $\sigma_i^N\in\{0,1\}$  the state of agent $i$ and  by  $\sigma^N:=(\sigma_i^N)_{i\geq1}$ be the configuration of the whole family.  Each agent, after an exponentially distributed time of parameter $1$, selects at random one agent in the population and adopts its state with a probability which depends on a function $\phi$ of the magnetization  $m^N=\frac{1}{N}\sum_{i=1}^N\sigma_i^N$, that represents a measure of the peer pressure felt by each agent.  
We assume that $\phi$ is a strictly positive $\mathcal{C}^1$ function. 
 The dynamics is described by a  $\{0,1\}^N$ valued continuous time Markov process $\{\usigma^N(t)\}_{t\geq0}=\{\big(\usigma_1^N(t),\ldots,\usigma_N^N(t)\big)\}_{t\geq0}$,  defined on a probability space $(\Omega,\mathcal{F},\{\mathcal{F}_t\},P)$,  
with transition rates for the $i^{th}$ component $\sigma_i^N$ given by:
\begin{equation}
 \label{rates1family}
c(i,\sigma^N)=\left\{\begin{array}{ll}
m^N\phi(m^N) & \mbox{if } \sigma_i^N=0,\\
(1-m^N)\phi(1-m^N) &  \mbox{if } \sigma_i^N=1.
\end{array} \right.
\end{equation} 
The generator of the process is therefore given by 
\begin{equation}\label{GeneratorMicroD}
\mathcal{L}_Nf\big(\sigma^N\big)=\sum_{i=1}^Nc(i,\sigma^N)[f(\sigma^{N,i})-f(\sigma^N)]
\end{equation}
where $f:\{0,1\}^N\to\mathbb{R}$ and $\sigma^{N,i}$ denotes the configuration obtained by 
$\sigma^N$ by replacing $\sigma_i^N$ with $1-\sigma_i^N$. 
We use the bold notation $\bm{m}^N:=\{\bm{m}^N(t)\}_{t\geq0}$ to denote the Markov process defined by $\bm{m}^N(t)=\frac{1}{N}\sum_{i=1}^N\usigma_i^N(t)$.  \\
As $N$ goes to infinity, assuming that a law large number holds, we expect the process $\big\{\bm{m}^N\ ; N>1\big\}$ to converge in distribution towards a deterministic process $m:= \{m(t)\}_{t\geq 0}$ which solves the  following equation 
\begin{equation}\label{macro one}
\dot{m} = m(1-m)\big[\phi(m)-\phi(1-m)\big].
\end{equation}
 
We have indeed the following result: 
\begin{prop}\label{prop: Kurtz one}	
Suppose there exists a non-random $\bar{m}\in[0,1]$ such that, for every $\epsilon>0$, 
\begin{equation*}
\lim_{N\to+\infty}P\big(|\bm{m}^N(0)-\bar{m}|>\epsilon\big)=0.
\end{equation*} Then, as $N\to +\infty$,  the sequence of Markov processes $\big\{ \bm{m}^N\ ; N>1\big\}$ converges in distribution, with respect to the Skorohod topology, to the unique solution of  equation (\ref{macro one}) with $m(0)=\bar{m}$.
\end{prop}  
\begin{proof}
Let $E^N=\{x\in [0,1]: x=\frac{j}{N},  0\leq j\leq N \}$ and  $\mathcal{L}_N$ be the generator defined in \eqref{GeneratorMicroD}. For $f:E^N\to\mathbb{R}$ we can write 
$f(m^N)=(f\circ g) (\sigma^N)$ with $g:\{0,1\}^N\to E^N$ given by $g(\sigma^N)=\frac{1}{N}\sum_{i=1}^N \sigma^N_i$. Then, we have $\mathcal{L}_N(f\circ g)(\sigma^N)=\mathcal{G}_N f(m^N)$, 
where
\begin{eqnarray*}
\mathcal{G}_Nf(x)&=&  Nx(1-x)\phi\left (x\right ) \left [ f\left (x+\frac{1}{N}\right )-f(x)
\right ] \\
&+& Nx(1-x)\phi\left(1-x\right ) 
 \left [ f
\left (x-\frac{1}{N}\right )-f(x)\right ].
\end{eqnarray*}
Denoting by $\mathcal{G}$ the generator of the semigroup associated to the evolution (\ref{macro one}), 
 by a simple computation we get
\begin{equation*}
\lim_{N\to+\infty}\sup_{x\in E^N}|\mathcal{G}_Nf(x)-\mathcal{G}f(x)|=0.
\end{equation*}
for any  $f\in\mathcal{C}^1([0,1])$.  
Then, by applying standard results on convergence of Markov processes (see, e.g., \cite{EthierKurtz}, Ch.\ 3, Corollary 7.4 and Ch.\ 4, Theorem 8.10) we obtain the desired result.
\end{proof}
We can observe that for any fixed $N$, by elementary theory of Markov processes,  the microscopic variable $\mathbf{m}^N$ has two absorbing states, $0$  and $1$, and it reaches one of them in finite time with probability one. On the other hand, depending on the solutions of equation $\phi(m)-\phi(1-m)$ in $[0,1]$, new equilibrium points, other that $0$ and $1$,  may appear in the macroscopic scale. 
Moreover, the stability of equilibria depends on the value of the derivative of  $\phi$ in such points and it may happen that the points $0$ and $1$ are unstable.	\\
For example, taking $\phi$ a linear function
\begin{equation*}
\phi(x)=ax+b
\end{equation*}
with $a\in\mathbb{R}$ and $b>\max\{-a,0\}$ ,  the interpretation of the role of $\phi$ in the dynamics is clearly determined by its derivative. If $a>0$, the  function is strictly increasing, so that  
the probability that one agent adopts a given opinion is proportional to how much that opinion is widespread in the population; we say in this case that agents have a \emph{cooperative} behaviour and. If $a<0$, the function is strictly decreasing and agents have a  \emph{competitive} behaviour, i.e., they have a bigger probability to adopt an opinion opposite to that of the majority. 
This behaviour is reflected in the stability  of the equilibrium points of the ODE \eqref{macro one}, see Figure \ref{fig: stability one fam}; if $a>0$ the  points $0$ and $1$ are asymptotically stable and the point $0.5$ is unstable while for $a<0$ the stability is exchanged.

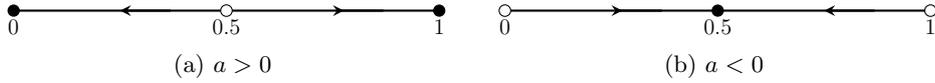
\begin{figure}[h!]
\centering
\subfloat[][$a>0$]{
\begin{tikzpicture}[shorten >=1pt, auto, scale=0.7 , transform shape,>=stealth]

\draw[thick](0,0) -- (8,0) node[anchor=north west] {};

\draw[] (0 cm,1pt) -- (0 cm,-1pt) node[anchor=north] {\large $0$};
\draw (4 cm,1pt) -- (4 cm,-1pt) node[anchor=north] {\large $0.5$};
\draw (8 cm,1pt) -- (8 cm,-1pt) node[anchor=north] {\large $1$};
\filldraw (0,0) circle (3pt);
\filldraw (8,0) circle (3pt);
\draw (4,0) node[circle, fill=white!100,draw,scale=0.6] {};
\draw [<-,thick] (2,0) -- (3,0);
\draw [>-,thick] (6,0) -- (7,0);

\end{tikzpicture}}
\quad
\subfloat[][$a<0$]{
\begin{tikzpicture}[shorten >=1pt, auto, scale=0.7 , transform shape,>=stealth]

\draw[thick](0,0) -- (8,0) node[anchor=north west] {};

\draw[] (0 cm,1pt) -- (0 cm,-1pt) node[anchor=north] {\large $0$};
\draw (4 cm,1pt) -- (4 cm,-1pt) node[anchor=north] {\large $0.5$};
\draw (8 cm,1pt) -- (8 cm,-1pt) node[anchor=north] {\large $1$};
\filldraw (4,0) circle (3pt);
\draw (0,0) node[circle, fill=white!100,draw,scale=0.6] {};
\draw (8,0) node[circle, fill=white!100,draw,scale=0.6] {};

\draw [>-,thick] (2,0) -- (3,0);
\draw [<-,thick] (6,0) -- (7,0);

\end{tikzpicture}}
\caption{Stability of the equilibrium points $0$,$1/2$ and $1$ for $\phi$ linear. }\label{fig: stability one fam}
\end{figure}

Notice that the particular choice of $a=0$ leads to the classical mean field voter model, so that the ODE \eqref{macro one} becomes $\dot{m}=0$ and the macroscopic system is frozen at its initial state. \\
An easy calculation shows that if $\phi$ is either strictly increasing or strictly decreasing we get the same phase diagram as for the linear case with $a>0$ and $a<0$ respectively. Then we shall refer to a \emph{cooperative} (respectively, \emph{competitive}) population if the gossip function is strictly increasing (respectively, decreasing). \\ 
Let us consider now the dynamics of the ''limiting particle''  $\bm{\upvarsigma}$, evolving through a time-inhomogeneous $\{0,1\}$-valued Markov process with jump rates
\begin{equation}\label{macrosigmaone}
c(i,\sigma, t)=\left\{\begin{array}{ll}
m(t)\phi(m(t)) & \mbox{if } \sigma_i=0,\\
(1-m(t))\phi(1-m(t)) &  \mbox{if } \sigma_i=1, 
\end{array} \right.
\end{equation}
where $\big(m(t)\big)$ is the solution to the differential equation \eqref{macro one}. Next proposition states that \emph{propagation of chaos} holds, i.e., as $N\to \infty$, particles behave independently according to the evolution of $\bm{\upvarsigma}$ defined in (\ref{macrosigmaone}) and  (\ref{macro one}). \\
More precisely, we recall that, given a polish space $(E, \rho)$,  for a sequence of  random vectors $\{X^N; N\geq 1\}$, where $X^N=(X^N_1, \ldots, X^N_N)\in E^N$  has a permutation invariant distribution for any $N$, 
we say that \emph{propagation of chaos holds} (or, equivalently, that the sequence is $\mu$-\emph{chaotic}),  if there exists a probability measure $\mu$ on $E$ such that, for any fixed $\nc\in \N$, we have 
$$(X^N_1, \ldots, X^N_{\nc})\stackrel{d}{\longrightarrow} (Y^N_1,\ldots, Y^N_{\nc})\quad \mbox{as } N\rightarrow +\infty,$$
where $Y^N_1,\ldots, Y^N_{\nc}$ are i.i.d. with common distribution $\mu$. \\
\begin{prop}\label{prop chaos one 0}
Let us fix $T>0$. For any $N>1$, let $\usigma^{N, T}=\{(\usigma^N_{1}(t),\ldots, \usigma^N_{N}(t)) \}_{t\in [0,T]}$ be the Markov process with generator $\mathcal{L}_N$ defined in \eqref{rates1family} and \eqref{GeneratorMicroD}.  Denote by $\mu_{[0,T]}$  the distribution of the Markov processes  $\usigma^T= \{\usigma(t)\}_{t\in [0,T]}$ with transition rates defined by \eqref{macrosigmaone} and \eqref{macro one} and by $\mu_0$ be the distribution of $\usigma(0)$.  Assume that $\usigma^{N, T}(0)$ has a permutation invariant distribution and  $\lim_{N\to\infty}E\big[ |\usigma_{i}^{N, T}(0)-\bar{\usigma}_{i}^N(0)|\big] =0$ for any $i$, where $\{\bar{\usigma}_{i}^N(0); i\geq 1\}$ are i.i.d. with common distribution $\mu_0$. Then, the sequence of stochastic processes $\{\usigma^{N, T}; N\geq 1\}$ is $\mu_{[0,T]}$-chaotic.\\
\end{prop}
 The proof can be  obtained as an application of Proposition 3.2 in \cite{Andreis} and it turns out to be an easy consequence of the analogous result for the delayed case, which will be proved in Proposition \ref{prop chaos one d} below.\\
 In next sections, we will resctrict for simplicity to the case when $\phi$ is strictly monotone, but all the results can be extended to a general function $\phi$ with localization arguments.

\subsection{Time spent close to macroscopic stable equilibria}
We have seen that, depending on the choice of the function $\phi$, new stable equilibrium points may appear in the thermodynamic limit. This suggests to investigate, for $N$ large but finite, the amount of time the microscopic variable  $\bm{m}^N$ spends near these equilibria during its transient phase. In Theorem \ref{teo: time stable} below, we consider the case when $\phi$ is a strictly monotone function, but  the result (with the obvious adjustments) clearly holds for any $\phi$ which is strictly monotone in a neighborhood of each equilibrium point, provided that the initial condition is chosen appropriately. We recall that  when $\phi$ is strictly increasing the ODE \eqref{macro one} has two asymptotically stable solutions, $0$ and $1$, while when $\phi$ is strictly decreasing the ODE \eqref{macro one} has only one asymptotically stable solution, $1/2$. For the first case we show that, whenever the initial condition $\bar{m}$ is smaller than $1/2$ (respectively, larger than $1/2$) with high probability the microscopic process  gets trapped in a small neighborhood of $0$ (respectively, $1$). For the second case we show that the microscopic process reaches a small interval containing $1/2$ in a finite time and remains close to it for a time exponentially large in $N$.
 
\begin{theorem}\label{teo: time stable}
For every $\epsilon>0$ and $N$  sufficiently large, 
\begin{itemize}
\item[i)] if $\phi$ is strictly increasing, then $\forall\ \bar{m}<\frac{1}{2}$ there exist $C_\epsilon >0$ and $T_\epsilon >0$ such that 
\begin{equation*}
P_{\bar{m}}\Big( \sup_{t\geq T_\epsilon}\bm{m}^N(t)>\epsilon\Big)<5N e^{-C_\epsilon N};
\end{equation*}
and, $\forall\ \bar{m}>\frac{1}{2}$ there exist $C_\epsilon >0$ and $T_\epsilon >0$ such that  
\begin{equation*}
P_{\bar{m}}\Big( \inf_{t\geq T_\epsilon}\bm{m}^N(t)<1-\epsilon\Big)<5N e^{-C_\epsilon N};
\end{equation*}
\item[ii)]  if $\phi$ it strictly decreasing, then $\forall\ \bar{m}\in (0,1)$ there exist $C_\epsilon >0,\ \widehat{C}_\epsilon>0$ and $T_\epsilon >0$ such that 
\begin{equation*}
P_{\bar{m}}\Big( \sup_{t\in[ T_\epsilon, T_\epsilon+e^{\widehat{C}_\epsilon N}]}|\bm{m}^N(t)-\frac{1}{2}|>\epsilon\Big)<14 N^2e^{-C_\epsilon N}.
\end{equation*}
\end{itemize}
\end{theorem}

\begin{proof}
Similarly to \cite{fagnani2016proc}, the argument of the proof uses Kurtz's Theorem and two auxiliary Lemmas, which are stated in the appendix. \\
We recall that $\{\bm{m}^N(t)\}_{t\geq 0}$ is a birth and death process with  values in $E_N=\{\frac{i}{N}: 0\leq i\leq N\}$ and with birth and death rates given respectively by $r^+(x)=Nx(1-x)\phi(x)$ and $ r^-(x)=Nx(1-x)\phi(1-x),\ x\in E_N$. 
Assume that $\phi$ is strictly decreasing and let us fix $\epsilon \in (0,\frac{1}{2})$ and 
 $N\geq  \frac{2}{\phi(\frac{1}{2})}$ so that $\mu:=\max_{x\in E_N}\big(r^+(x)+r^-(x)\big)\geq 1$.
 In this case, letting $\bar{\epsilon}=\frac{\epsilon}{3}$ we have, for any $x<\frac{1}{2}-\bar{\epsilon}$,   $$\frac{r^+(x)}{r^-(x)}=\frac{\phi(x)}{\phi(1-x)}\geq \frac{\phi(\frac{1}{2}-\bar{\epsilon})}{\phi(\frac{1}{2}+\bar{\epsilon})}\geq 1+\delta$$ 
and, for any $x>\frac{1}{2}+\bar{\epsilon}$
 $$\frac{r^-(x)}{r^+(x)}=\frac{\phi(1-x)}{\phi(x)}\geq \frac{\phi(\frac{1}{2}-\bar{\epsilon})}{\phi(\frac{1}{2}+\bar{\epsilon})}\geq 1+\delta$$ 
for a suitable $\delta >0$. 
 Then we can apply Lemma \ref{lem: BD unmezzo} with $x_0=\frac{1}{2}-2\bar{\epsilon}$ first to the case $\bm{X}(t)=\bm{m}^N(t)$ and then to the case $\bm{X}(t)=1-\bm{m}^N(t)$ to obtain respectively \\
\begin{eqnarray*}
P_x\left( \inf_{t\in [0, e^{\bar{\epsilon}C_p N}]}\bm{m}^N(t)<\frac{1}{2}-3\bar{\epsilon}\right ) < 10 N^2 e^{-\bar{\epsilon}C_p N} \quad \forall x> \frac{1}{2}-2\bar{\epsilon}\\
P_x\left( \sup_{t\in [0, e^{\bar{\epsilon}C_p N}]}\bm{m}^N(t)>\frac{1}{2}+3\bar{\epsilon}\right ) < 10 N^2 e^{-\bar{\epsilon}C_p N} \quad \forall x<\frac{1}{2}+2\bar{\epsilon}
\end{eqnarray*} 
 from which it follows that, for any $\epsilon\in (0, \frac{1}{2})$
\begin{equation}\label{D_lemma2}
P_x\left( \sup_{t\in \left[0, e^{\frac{\epsilon}{3} C_p N}\right]}\left|\bm{m}^N(t)-\frac{1}{2}\right|>\epsilon \right ) < 10 N^2 e^{-\frac{\epsilon}{3}C_p N} \quad \forall x\in I_\epsilon=\left(\frac{1}{2}-\frac{2}{3}\epsilon, \frac{1}{2}+\frac{2}{3}\epsilon\right).\\
\end{equation}
Observe that, for any $m(0)\in (0,1)$,  the solution of \eqref{macro one} converges to $1/2$ as $t\to +\infty$, then we can find $T_\epsilon$ such that 
$m(T_\epsilon)\in \left(\frac{1}{2}-\frac{\epsilon}{3}, \frac{1}{2}+\frac{\epsilon}{3} \right)$ and, letting $\bar{m}=m(0)$, we can write 
\begin{multline}\label{ineq}
P_{\bar{m}}\left( \sup_{t\in \left[T_\epsilon, T_\epsilon+ e^{\frac{\epsilon}{3}C_p N}\right]}\big|\bm{m}^N(t)-\frac{1}{2}\big|>\epsilon\right)\leq \\ P_{\bar{m}}\left( \sup_{t\in [T_\epsilon, T_\epsilon+e^{\frac{\epsilon}{3}C_p N}]}\big|\bm{m}^N(t)-\frac{1}{2}\big| >\epsilon\  \bigg|\  \bm{m}^N(T_\epsilon) \in I_\epsilon \right) +P_{\bar{m}}\Big(
\bm{m}^N(T_\epsilon) \notin I_\epsilon \Big).
\end{multline}
Applying Kurtz's Theorem \ref{teo:kurtz} we have $$P_{\bar{m}}(\bm{m}^N(T_\epsilon)\notin I_\epsilon)\leq P_{\bar{m}}\left(\sup_{t\in [0, T_\epsilon]}|\bm{m}^N(t)-m(t)|>\frac{\epsilon}{3}\right)<4 e^{-\left(\frac{\epsilon}{3}\right)^2 
C_{T_\epsilon}N}.$$ Then, letting $\widehat{C}_\epsilon=\frac{\epsilon}{3}C_p$ and  $C_\epsilon=\min\{\frac{\epsilon}{3}C_p,\  \left(\frac{e}{3}\right)^2 C_{T_\epsilon}\}$, 
 using \eqref{D_lemma2} and the Markov property for the first term of \eqref{ineq} 
 we get the desired result.\\
%$\leq 14 N^2 e^{- C_\epsilon N}$ 
Now, assume that $\phi$ is strictly increasing and fix $\epsilon\in (0,\frac{1}{2})$. Then, for any $x\in (0, \epsilon]$ we have $\frac{r^-(x)}{r^+(x)}=\frac{\phi(1-x)}{\phi(x)}>\frac{\phi(1-\epsilon)}{\phi(\epsilon)}=1+\delta$  for some $\delta>0$, and applying Lemma \ref{lem: BD zero} we get that, for any $x<\frac{\epsilon}{2}$, 
$$P_{x}\left(\sup_{t\geq 0} \bm{m}^N>\epsilon\right)<\frac{\epsilon}{2}e^{-\frac{\epsilon}{2}CN}.$$ 
Then, observing that, for any $m(0)< \frac{1}{2}$ we have $\lim_{t\to +\infty}m(t)=0$ and choosing $T_\epsilon$ such that $m(T_\epsilon)<\frac{\epsilon}{4}$, we can use the same argument as above and, for any $\bar{m}=m(0)<\frac{1}{2}$, we get 
$$P_{\bar{m}}\left(\sup_{t\geq T_\epsilon}\bm{m}^N(t)>\epsilon\right)< 5N e^{-C_\epsilon N}$$  
for a suitable constant $C_\epsilon$. 
 Finally, the second inequality in statement \emph{ii)} can be obtained from the first one by considering the process $\{1-\bm{m}^N(t)\}_{t\geq 0}$.
\end{proof}
\begin{remark}
Notice that, as an immediate consequence of  part $ii)$ of Theorem \ref{teo: time stable}, we get that in the case of a competitive population  the mean absorption time for the microscopic process is at least exponential in the system's size. 
Indeed, letting $\tau$ be such time, it is enough to fix any $\epsilon>0$ and $  \bar{m}\in (0,1)$ and observe that $E_{\bar{m}}[\tau]\geq e^{\widehat{C}_\epsilon N}
P_{\bar{m}}(\sup_{t\in [T_\epsilon, T_\epsilon+e^{\widehat{C}_\epsilon N}]}|\bm{m}^N(t)-\frac{1}{2}|\leq \epsilon)$.\\
Observe that, instead, for the classical mean field voter model (which corresponds to our model with $\phi$ constant) the dependence of absorption time on the size of the system is of power law type. 
\end{remark}
\section{Large scale dynamics in the presence of delay: emergence of periodic behaviours}
\subsection{Definition, thermodynamic limit and chaos propagation}\label{section:delay}
Let now suppose that each agent gets information on the state of the system with a certain delay.  We  modify the model 
in such a way that the influence of the magnetization on the dynamics is weighted through a delay kernel of Erlang type, as done in  \cite{ditlevsen2017multi} and \cite{collet2016rhythmic}.\\
For $k\in\mathbb{R}_+$, $n\in\mathbb{N}$ fixed, the information available to each agent at time $t$ is represented by the variables 
\begin{eqnarray*}
\ugamma^{(N, n)}(t):=\int_0^t \bm{m}^N(s)\frac{(t-s)^n}{n!}k^{n+1}e^{-k(t-s)}ds,\\
\ueta^{(N, n)}(t):=\int_0^t \big(1-\bm{m}^N(s)\big)\frac{(t-s)^n}{n!}k^{n+1}e^{-k(t-s)}ds,
\end{eqnarray*}
The parameters $n$ and $k$ tune the weight given to the past configurations, more precisely the parameter $k$ tells us how close to the current time $t$ is the maximum weight of the delay and the parameter $n$ gives us information on the shape of the peak.\\
The evolution of the opinion of an agent $i$ is given by the following time inhomogeneous transition rates:
\begin{equation*}
c (i,\sigma^N, t)=\left\{\begin{array}{ll}
m^N(t)\phi\big(\gamma^{(N,n)}(t)\big) & \mbox{if } \sigma_i^N=0,\\
\big(1-m^N (t))\phi\big(\eta^{(N,n)}(t)\big) &  \mbox{if } \sigma_i^N=1.
\end{array} \right.
\end{equation*} 
The advantage of using Erlang kernels is that, by increasing the number of variables, the system becomes markovian: indeed let us define, for $j=0,\ldots, n$ the processes $\ugamma^{(N, j)}:=\{\ugamma^{(N,j)}(t)\}_{t\geq0}$ and $\ueta^{(N, j)}:=\{\ueta^{(N,j)}(t)\}_{t\geq0}$ with 
\begin{eqnarray*}
\ugamma^{(N,j)}(t):=\int_0^t \bm{m}^N(s)\frac{(t-s)^j}{j!}k^{j+1}e^{-k(t-s)}ds,\\
\ueta^{(N,j)}(t):=\int_0^t \big(1-\bm{m}^N(s)\big)\frac{(t-s)^j}{j!}k^{j+1}e^{-k(t-s)}ds. 
\end{eqnarray*}
Letting  $\ugamma^{N}:=\big(\ugamma^{(N,j)}\big)_{j=0}^n$ and  $\ueta^{N}:=\big(\ueta^{(N,j)}\big)_{j=0}^n$ and denoting by $\gamma^N=
\big(\gamma^{(N,j)}\big)_{j=0}^n,\ \eta^N=\big(\eta^{(N,j)}\big)_{j=0}^n\in \R^{n+1}$ their possible states, the system is completely described by the $N+2(n+1)$ dimensional process 
\begin{equation*}
\Big\{\big(\usigma^N(t), \ugamma^{N}(t),\ueta^{N}(t)\big)\Big\}_{t\geq0}
\end{equation*}
which is a Markov process with generator 
\begin{multline}\label{generator micro delay}
\mathcal{L}_N^df(\sigma^N, \gamma^N,\eta^N) = \sum_{i=1}^N c\big(i,\sigma^N, \gamma^{(N,n)},\eta^{(N,n)}\big)\big[f(\sigma^{N,i}, \gamma^N,\eta^N)-f(\sigma^N, \gamma^N,\eta^N)\big] \\
+ \sum_{j=0}^n \Big[ k\big(\gamma^{(N,j-1)}-\gamma^{(N,j)}\big)\frac{\partial f}{\partial\gamma^{(N,j)}}+k\big(\eta^{(N,j-1)}-\eta^{(N,j)}\big)\frac{\partial f}{\partial\eta^{(N,j)}} \Big]
\end{multline}
where  
 $$\textstyle c (i,\sigma^N, t)=\left\{\begin{array}{ll}
m^N \phi(\gamma^{(N,n)}) & \mbox{if } \sigma_i^N=0,\\
\big(1-m^N)\phi(\eta^{(N,n)}) &  \mbox{if } \sigma_i^N=1
\end{array} \right.,$$ $f$ is a smooth function on $\{0,1\}^N\times[0,1]^{n+1}\times[0,1]^{n+1}$,  $\sigma^{N,i}$ denotes the configuration obtained from $\sigma^N$ by replacing $\sigma_i^N$ with $1-\sigma_{i}^N$ and 
we have used the convention 
\begin{equation*}
\gamma^{(N,-1)}:=m^N\quad\eta^{(N,-1)}:=1-m^N.
\end{equation*}

\begin{remark}
Observe that, if for some time $t'>0$ the magnetization $\bm{m}^N(t')\in\{0,1\}$ then agents will not change their state for any $t>t'$ and for all $j$, the process $\big(\ugamma^{(N,j)}(t)$,$\ueta^{(N,j)}(t)\big)$  will converge, as $t\to\infty$, to either $(0,1)$ or $(1,0)$ according to the value of $\bm{m}^N(t')$. 
\end{remark}

As in the case without delay, we study the limiting dynamics as $N$ goes to infinity. Next proposition states convergence in distribution of the sequence of microscopic processes $\{\big(\bm{m}^N, \bm{\gamma}^N, \bm{\eta}^N\big); N\geq 1\}$ to the solution  to the following non linear system of ODEs:
\begin{equation}\label{eq:macro}
\left\{\begin{array}{llllllll}
\dot{m}=m(1-m)[\phi(\gamma^{(n)})-\phi(\eta^{(n)})],\\
\dot{\gamma}^{(n)}=k[\gamma^{(n-1)}-\gamma^{(n)}],\\
\dot{\eta}^{(n)}=k[\eta^{(n-1)}-\eta^{(n)}],\\
\ldots\quad \ldots\quad \ldots\\
\ldots\quad \ldots\quad \ldots\\
\dot{\gamma}^{(0)}=k[m-\gamma^{(0)}],\\
\dot{\eta}^{(0)}=k[1-m-\eta^{(0)}],\\
\end{array}\right.
\end{equation}
with $\gamma^{(j)}(t)=\int_0^t m(s)\frac{(t-s)^j}{j!}k^{j+1}e^{-k(t-s)}ds$ and 
$\eta^{(j)}(t)=\int_0^t (1-m(s))\frac{(t-s)^j}{j!}k^{j+1}e^{-k(t-s)}ds$, $\forall j=0,\ldots,n$.\\
\begin{prop}
Assume that, given $(\bar{m},\bar{\gamma},\bar{\eta})\in[0,1]^{2n+3}$, for  every $\epsilon>0$,
 \begin{equation*}
 \lim_{N\to+\infty}P\Big(\|\big(\bm{m}^N(0),\ugamma^N(0),\ueta^N(0)\big)-(\bar{m},\bar{\gamma},\bar{\eta})\|>\epsilon\Big)=0
 \end{equation*}
 where $\|v\|$ denotes the norm of the vector $v$.  
 Then, as $N\to +\infty$, the sequence of Markov processes $\big\{(\bm{m}^N,\ugamma^N,\ueta^N);N\geq1\big\}$ converges in distribution, with respect to the Skorohod topology, to the unique solution of the system of equations \eqref{eq:macro} with $\big(m(0),\gamma(0),\eta(0)\big)=(\bar{m},\bar{\gamma},\bar{\eta})$.
\end{prop}
\begin{proof}
Following the proof of Proposition \ref{prop: Kurtz one} observe that, for any  
$f: E^N\times [0,1]^{n+1}\times [0,1]^{n+1}$ we have $f(m^N, \gamma^N, \eta^N)=(f\circ g)(\sigma^N, \gamma^N, \eta^N)$ and 
$
\mathcal{L}_N^d (f\circ g)(\sigma^N, \gamma^N,\eta^N) = \mathcal{G}_N^d f(m^N, \gamma^N,\eta^N)$  
with 
\begin{multline}
\mathcal{G}_N^d f(m^N, \gamma^N,\eta^N)=Nm^N(1-m^N)\phi(\eta^{(N,n)})\big[ f(m^N-1/N,\gamma^{(N,n)},\eta^{(N,n)})\\
  -f(m^N,\gamma^{(N,n)},\eta^{(N,n)})\big]\\
 + Nm^N(1-m^N)\phi(\gamma^{(N,n)})\big[ f(m^N+1/N,\gamma^{(N,n)},\eta^{(N,n)})f(m^N,\gamma^{(N,n)},\eta^{(N,n)})\big] \\
 + \sum_{j=0}^n k\big(\gamma^{(N,j-1)}-\gamma^{(N,j)}\big)\frac{\partial f}{\partial\gamma^{(N,j)}}+k\big(\eta^{(N,j-1)}-\eta^{(N,j)}\big)\frac{\partial f}{\partial\eta^{(N,j)}}. \nonumber
\end{multline}    
Then the conclusion follows as in Proposition \ref{prop: Kurtz one} taking $\mathcal{G}^d$ the generator of the semigroup associated to the deterministic evolution \eqref{eq:macro}. 
\end{proof}
Next result is propagation of chaos: if $N$ is large enough agents in any finite group behave like independent units, each one following the evolution of a limit process $\usigma$ with jump rates:
\begin{equation}\label{macrosigmaonedelay}
\begin{array}{ll}
0\to 1\quad & m\phi\big(\gamma^{(n)}\big),\\
1\to 0\quad &(1-m)\phi\big(\eta^{(n)}\big)
\end{array}
\end{equation}
where $m$, $\gamma^{(n)}$ and $\eta^{(n)}$ are the solutions of \eqref{eq:macro}. \\

\begin{prop}\label{prop chaos one d}
Let us fix $T>0$. For any $N>1$, let $(\usigma^{N,T}, \ugamma^{N,T}, \ueta^{N,T})$ $=$ $\{\big(\usigma^N(t), \ugamma^N(t),$ $ \ueta^N(t)\big) \}_{t\in [0,T]}$ be the Markov process with generator $\mathcal{L}_N^d$ defined in \eqref{generator micro delay}.  Denote by $\mu_{[0,T]}$  the distribution of the Markov processes  $\usigma^T= \{\usigma(t)\}_{t\in [0,T]}$ with transition rates defined in \eqref{macrosigmaonedelay} and \eqref{eq:macro} and by $\mu_0$ the distribution of $\usigma(0)$.  Assume that $\usigma^{N, T}(0)$ has a permutation invariant distribution and  $\lim_{N\to\infty}E\big[ |\usigma_{i}^{N, T}(0)-\bar{\usigma}_{i}^N(0)|\big] =0$ for any $i$, where $\{\bar{\usigma}_{i}^N(0); i\geq 1\}$ are i.i.d. with common distribution $\mu_0$. Then, the sequence of stochastic processes $\{\usigma^{N, T}; N\geq 1\}$ is $\mu_{[0,T]}$-chaotic.\\
\end{prop}

 \begin{proof}
 The proof uses standard arguments on propagation of chaos (see, e.g., \cite{sznitman1991topics},\cite{graham1992mckean}, \cite{Andreis}).\\
 We recall that, given a polish space $(E, \rho)$ and the set $\mathcal{M}_1(E)$ of Borel probability measures on $E$ with finite first moment,  the Wasserstein distance between  $\mu, \nu\in \mathcal{M}_1(E)$ is defined by 
 \begin{equation}\label{wasserstein}
 W^1_\rho(\mu, \nu)=\inf\{\int_{E\times E}\rho(x,y)\pi(dx, dy): \pi \mbox{ has marginals } \mu \mbox{ and } \nu\}
 \end{equation}
 and that convergence with respect to such distance implies weak convergence.  
Now, let us fix $\nc\in \N$. For $N\geq \nc$,  denote by 
 $\mu^{N, \nc}_{[0,T]}$ the distribution of the process $\{(\usigma^N_{1}(t),\ldots, \usigma^N_{\nc}(t)) \}_{t\in [0,T]}$ and by $\mu_{[0,T]}^{\otimes \nc}$ the product of $\nc$ copies of $\mu_{[0,T]}$. Such measures are defined on the space  $\mathcal{D}([0,T]; \mathbb{R}^\nc)$ of $\R^\nc$- valued  càdlàg  functions on $[0, T]$ with the Skorohod distance $\rho$. \\
 
We prove the following stronger result: 
$$\lim_{N\to +\infty}W^1_{\bar{\rho}} (\mu^{N,\nc}_{[0,T]}, \mu_{[0,T]}^{\otimes \nc})=0,$$ 
where $\bar{\rho}=\| \cdot \|_{\infty}$ denotes the uniform metric on 
 $\mathcal{D}([0,T]; \mathbb{R}^\nc)$ and $W^1_{\bar{\rho}}$ is defined as in \eqref{wasserstein}.
To this purpose, we define a suitable coupling between the microscopic system and a system of i.i.d.\ particles each one following the limiting dynamics.\\
Let us fix a  probability space $(\Omega, \mathcal{F}, \{\mathcal{F}_t\},  P)$ satisfying the usual conditions, where it is defined  a family
 $\mathcal{N}=\{ \mathcal{N}^{i};  i\geq 1\}$ 
 of i.i.d.\ adapted Poisson random measures with intensity  
 $\ell \otimes \ell$, where $\ell$ denotes the restriction to $[0, \infty )$
  of the Lebesgue measure.
 The microscopic process $(\usigma^N, \ugamma^N, \ueta^N)$ can be realized as the solution of the following SDE with initial condition $\usigma^N(0)=\usigma^{N,T}(0)$:
\begin{eqnarray}\label{eq:SDEdelaySigma}
&& \de \bm{\upvarsigma}_{i}^N(t) =\int_0^\infty\psi( \bm{\upvarsigma}_{i}^N(t-))\mathbbm{1}_{(0, \zeta( \bm{\upvarsigma}_{i}^N(t-),\bm{m}^N(t-),\ugamma^{(N,n)}(t-),\ueta^{(N,n)}(t-))]}(u)\mathcal{N}^{i}(\de u,\de t), \\\label{eq:SDEdelayGamma}
&& \de \ugamma^{(N,j)}(t) = k\big(\ugamma^{(N,j-1)}(t)-\ugamma^{(N,j)}(t)\big)\de t,\\\label{eq:SDEdelayEta}
&& \de \ueta^{(N,j)}(t) = k\big(\ueta^{(N,j-1)}(t)-\ueta^{(N,j)}(t)\big)\de t,
\end{eqnarray}
where $i=1,\ldots,N$, $j=0,\ldots,n$,  
 $\zeta:\{0,1\}\times[0,1]\times[0,1]^2\to\mathbb{R}_+$ is the jump rate function
\begin{equation*}
\zeta( \sigma_{i}^N,m^N,\gamma^{(N,n)},\eta^{(N,n)}) = (1-\sigma_i^N)m^N\phi(\gamma^{(N,n)})+\sigma_i^N(1-m^N)\phi(\eta^{(N,n)}) 
\end{equation*}
and $\psi:\{0,1\}\to\mathbb{R}_+$ is the jump amplitude function
\begin{equation*}
\psi(\sigma_i^N)=1-2\sigma_i^N.
\end{equation*}  
Strong existence and uniqueness of solutions of \eqref{eq:SDEdelaySigma}, \eqref{eq:SDEdelayGamma}, \eqref{eq:SDEdelayEta} can be easily derived from  Theorem 1.2 of \cite{graham1992mckean}: indeed, 
equations \eqref{eq:SDEdelayGamma}, \eqref{eq:SDEdelayEta} are linear and,  letting 
$f(\sigma^N_{i}, \gamma^N, \eta^N,  u):=\psi (\sigma^N_{i})\mathbbm{1}_{(0,\lambda( \sigma_{i}^N, m(\sigma^N), \gamma^{N, n}, \eta^{N. n})]}(u)$, where $m(\sigma^N)=\frac{1}{N}\sum_{j=1}^{N}\sigma^N_{j}$, recalling that $\phi$ is Lipschitz, the following $L^1$ Lipschitz condition holds:  
$$\sum_{i=1}^N \int |f(\sigma^N_{i}, \gamma^N, \eta^N,  u)-f(\tilde{\sigma}^N_{i}, \tilde{\gamma}^N, \tilde{\eta}^N, u)| \de 	u\leq C  
 \|(\sigma^N, \gamma^N, \eta^N)-(\tilde{\sigma}^N, \tilde{\gamma}^N, \tilde{\eta}^N)\| $$ for all $(\sigma^N, \gamma^N, \eta^N),\   (\tilde{\sigma}^N,  \tilde{\gamma}^N, \tilde{\eta}^N)  \in \{0,1\}^N\times [0,1]^{n+1}\times [0,1]^{n+1}$,
where $\|\cdot \|$ denotes the $\ell^1$ norm on $\mathbb{R}^N\times \mathbb{R}^{2n+2}$ and $C$ is a suitable constant.\\
On the same probability space, let us consider a vector of i.i.d. processes $\bar{\usigma}^N:=(\bar{\usigma}^N_i)_{i=1}^{N}$ with evolution defined by:
\begin{equation}\label{eq:SDEdelaySigmaM}
\de \bar{\usigma}_{i}^N(t)  = \int_0^\infty\psi( \bar{\usigma}_{i}^N(t-))\mathbbm{1}_{(0,\zeta( \bar{\usigma}_{i}^N(t-),m(t-),\gamma^{(n)}(t-),\eta^{(n)}(t-))]}(u)\mathcal{N}^{i}(\de u,\de t)
\end{equation}
for $i=1,\ldots, N$, where $m, \gamma^{(n)},\eta^{(n)}$ are the solutions to the \eqref{eq:macro} and the family of Poisson random measure $\mathcal{N}=\{\mathcal{N}^i;\ i=1,\ldots,N\}$ is the same as the one in equations \eqref{eq:SDEdelaySigma}.\\ 
Now, by the above properties we get, for any $i=1,\ldots,\nc$  
\begin{multline}
E\left[\sup_{r\in[0,t]}\big|\usigma_{i}^N(r)-\bar{\usigma}_i^N(r)\big|\right]\leq C_1\int_0^t E\left[\sup_{r\in[0,s]}\big|\usigma_{i}^N(r)-\bar{\usigma}_i^N(r)\big|\right]\de s \\ 
+ E\left[\big|\usigma_{i}^N(0)-\bar{\usigma}_i^N(0)\big|\right]\ + C_2 \int_0^t E\left[\big|\bar{\bm{m}}^N(s)-m(s)\big|\right]\de s \\
+ C_3 \int_0^t \int_0^sE\left[\big|\bar{\bm{m}}^N(r)-m(r)\big|\right]\frac{(s-r)^n}{n!}k^{n+1}e^{-k(s-r)}\de r\de s,
\end{multline}
where $C_1,C_2,C_3$ depend only on $\|\phi\|_\infty$, $\|\phi'\|_\infty, k$ and $n$. \\
Now, using the hypothesis on the initial conditions and the law of large numbers for the sequence $\{\bar{\bm{m}}^N;N\geq 1\}$, by Gronwall's Lemma   
we can conclude that
 $$E[\bar{\rho}(\usigma^N,\ \bar{\usigma}^N)]\leq \sum_{i=1}^\nc E[\sup_{t\in [0, T]}|\usigma^N_i(t)-\bar{\usigma}^N_i(t)|]\xrightarrow{N\to \infty} 0$$
and the proof is complete. 
 \end{proof}     

\subsection{Delay induced Hopf bifurcation in competitive environments}
The purpose of this section is to investigate  the emergence of periodic solutions of the system of ODEs \eqref{eq:macro} when the delay parameters are properly taken. We recall that points in the phase space are vectors of the form 
 $(m, \gamma^{(0)},\ldots,\gamma^{(n)},\eta^{(0)}\ldots,\eta^{(n)})\in \R^{2n+3}$.\\ 
If $\phi$ is strictly monotone we can easily observe that the system has three equilibrium points $x_0=(0,0,\ldots,0,1,\ldots,1)$, $x_1=(1,1,\ldots,1,0,\ldots,0)$ and $x^*=(\frac{1}{2},\frac{1}{2},\ldots,\frac{1}{2}, \frac{1}{2}\ldots,\frac{1}{2})$ in $\R^{2n+3}$. The points $x_0$ and $x_1$ are traps for the microscopic process, while the equilibrium point $x^*$ appears only in the macroscopic model. We recall that in the model without delay a phase transition occurs when the sign of the derivative of $\phi$ changes, so that the stability  
of the point $\frac{1}{2}$ gets modified. Here we a fix a strictly monotone function $\phi$ in the delayed model, and look for a phase transition 
corresponding to the delay parameters $n$ and $k$ (as defined in section \ref{section:delay}). More precisely, focusing on the case when $\phi$ is decreasing,   
we aim at showing that, for at least one value of $n$, corresponding to the parameter $k$ there is a supercritical Hopf bifurcation at $x^*$, that is, for $k$ below a certain threshold, the point $x^*$ becomes unstable and a stable limit cycle bifurcates from it.\\

We start linearising  the system \eqref{eq:macro} for a fixed $n\geq1$. The Jacobian matrix $J(k)\in\mathbb{R}^{2n+3}\times\mathbb{R}^{2n+3}$ at the equilibrium point $x^*$  is given by
\begin{equation}\label{macrogeneral}
J(k)=\left(\begin{array}{lllllllllccccccccc}
0 & \frac{\phi'(\frac{1}{2})}{4} & \frac{-\phi'(\frac{1}{2})}{4} & 0 & 0 & \ldots & 0 & 0 \\ 
0 & -k & 0 & k & 0 &\ldots & 0 & 0 \\
0 & 0  & -k & 0 &  k &\ldots & 0 & 0 \\
0 &  0 & 0 & -k &  0 & \ldots & 0 & 0 \\
. & . & . & . & . & \ldots& . & .\\
. & . & . & . & . & \ldots & . & . \\
k & 0 & 0 & 0 & 0 & \ldots &-k & 0 \\
-k & 0 & 0 & 0 & 0& \ldots & 0 &-k \\
\end{array}
\right )
\end{equation} 
We can immediately observe that only the derivative in $\frac{1}{2}$ of the function $\phi$ appears in the matrix. Letting $c:=\frac{\phi'(\frac{1}{2})}{4}$,  the associated characteristic equation is given by:
\begin{equation}\label{eq:charPolyG}
	 (\lambda+k)^{n+1}\big(\lambda (\lambda+k)^{n+1}-2ck^{n+1}\big)=0
\end{equation}

Observe that if $\phi$ is increasing, i.e. if $c>0$, whatever the values of $k$ and $n$ are, there exist at least one positive eigenvalue $\lambda^+$ and one negative eigenvalue $\lambda^-$. Indeed 
letting $g(\lambda)= (\lambda+k)^{n+1}\big(\lambda (\lambda+k)^{n+1}-2ck^{n+1}\big)$ we have $g(-k)=0$, and so $\lambda^- =-k$, and 
$g(0)=-2ck^{n+1}<0$, which implies that there exists $\lambda^+\in\mathbb{R}_+$ such that $g(\lambda^+)=0$. Then $x^*$ is unstable for any $k$ and $n$. 
\\ 
Then, let us consider $\phi$ decreasing.  In this case, since  $c<0$, the previous argument does not apply and, due to the impossibility, to our knowledge, of writing an explicit expression for the solutions to the characteristic equation \eqref{eq:charpoly} we cannot use  standard arguments to prove the existence of a Hopf bifurcation (see, e.g., Theorem 1 pag 352 in \cite{perko2013differential}). 
Then we shall use the following criterion, stated in \cite{LiuHopf}, for detecting Hopf bifurcation without using eigenvalues. \\

Let consider a one-parameter family of differential equations
\begin{equation}\label{eq:diffsystem}
\dot{\chi}(t)=F_k\big(\chi(t)\big),\qquad t\geq 0,\ \chi(t)\in\mathbb{R}^n,\ k\in\mathbb{R}, 
\end{equation}
with an equilibrium point $(x^*,k^*)$ and $F_k\in\mathcal{C}^\infty\big(\mathbb{R}^n\big)$, $\forall k\in\mathbb{R}$. 
Let denote the characteristic polynomial of the Jacobian matrix $J(k):=\big(\frac{\partial}{\partial z_i}(F_{k})_j\big)_{i,j=1,\ldots n}$ by
\begin{equation*}
p(\lambda;k)=p_n(k)\lambda^n+\ldots+p_1(k)\lambda+p_0(k)
\end{equation*}
where every $p_i(k)$ is a smooth function of $k$ and $p_n(k)=1$. For any $j=1,\ldots,n$ define the matrices
\begin{equation}\label{Ln}
L_j(k)=\left(\begin{array}{llllcccc}
p_1(k) & p_0(k) & \cdots & 0 \\ 
p_3(k) & p_2(k) & \cdots & 0 \\ 
\cdot & \cdot &  \cdots & \cdot\\
\cdot & \cdot &  \cdots & \cdot\\
p_{2j-1}(k) & p_{2j-2}(k) &\cdots & p_n(k)	 
\end{array}
\right ),
\end{equation}
where $p_i(k)=0$ if $i<0$ or $i>n$ and call $D_j(k)$ their determinant, i.e. $D_j(k):=\det\big(L_j(k)\big)$.  

\begin{theorem}\label{teo:criterionH}(\cite{LiuHopf})
Assume there is a smooth curve of equilibria $\big(x^k,k\big)$ with $x^{k^*}=x^*$ for the system of differential equations \eqref{eq:diffsystem}. If 
\begin{eqnarray}
(CH1) & p_0(k^*)>0, D_1(k^*)>0,\ldots,D_{n-2}(k^*)>0,D_{n-1}(k^*)=0 \\
(CH2) & \frac{d D_{n-1}(k^*)}{dk}\neq 0
\end{eqnarray}
then there exists a supercritical Hopf bifurcation.
\end{theorem}

In a model like the one we consider, the above criterion is analytically easy to apply only for small values of $n$ 
In any case, our interest in detecting delay-induced periodicity is purely qualitative, therefore we will consider only the case $n=2$. In the case when 
$\phi$ is a linear function the dimension of the system can be reduced and we are able to get the result analytically also for $n=3,4,5$,  with an explicit expression of the associated critical parameter $k^*$ (see Appendix \ref{appendixA}). 
However, we believe that using numerical methods, the result can be shown 
to hold for bigger values of $n$ and for a general $\phi$.\\  
Let us first analyse the case when $\phi$ is linear function and we can find an explicit value for the bifurcation parameter $k^*$.
%Then we consider the case  of a general $\phi$, where an explicit formula for $k^*$ can be found using a mathematical software. We will illustrate the emergence of a periodic behaviour in both cases also through simulations of the microscopic process.  In the appendix we extend the calculation of the linear case for $n=3,4,5$ for which we provide an explicit expression of $k^*$. 

Let  $\phi(z)=-az+b$, with $b>-a$ and $a>0$,  The system \eqref{eq:macro} can be rewritten introducing new variables $\xi^{(i)}=\gamma^{(i)}-\eta^{(i)}$, $i=0,\ldots, n$ as follows  
\begin{equation}\label{macro2}
\left\{\begin{array}{llllllll}
\dot{m}=-am(1-m)\xi^{(n)};\\
\dot{\xi}^{(n)}=k[\xi^{(n-1)}-\xi^{(n)}];\\
\ldots\quad \ldots\quad \ldots\\
\ldots\quad \ldots\quad \ldots\\
\dot{\xi}^{(0)}=k[2m-\xi^{(0)}-1];\\
\end{array}\right.
\end{equation}
The equilibrium points are  $x_0=(0,-1,\ldots,-1)$, $x_1=(1,1,\ldots,1)$ and $x^*=(\frac{1}{2},0,\ldots,0)$.  Linearising the system at $x^*$, the Jacobian matrix $J(k)\in\mathbb{R}^{2n+3}\times\mathbb{R}^{2n+3}$ has  the following form
\color{black}
\begin{equation}\label{matrix:linear}
J(k)=\left(\begin{array}{llllllllcccccccc}
0 & -\frac{a}{4} & 0 & 0 & \ldots & 0 & 0 \\ 
0 & -k & k & 0 & \ldots & 0 & 0 \\
0 &  & -k & k & \ldots & 0 & 0 \\
. & . & . & . &\ldots& . & .\\
. & . & . & . &\ldots & . & . \\
0 & 0 & 0 & 0 & \ldots &-k & k \\
2k & 0 & 0 & 0 & \ldots & 0 &-k \\
\end{array}
\right )
\end{equation}
The associated characteristic equation is given by 
\begin{equation}\label{eq:charpoly}
\lambda (\lambda+k)^{n+1}+\frac{a}{2}k^{n+1}=0.
\end{equation}
Now using the notation in \eqref{eq:diffsystem} we have $\chi=(m,\xi^{(0)},\ldots,\xi^{(n)})$ and $F_k(\chi)=\big(-am(1-m)\xi^{(n)},k[2m-\xi^{(0)}-1],\ldots,k[\xi^{(n-1)}-\xi^{(n)}]\big)$. 

For $n=2$, according to formula \eqref{eq:charpoly} and applying 
Theorem \ref{teo:criterionH} we obtain
\begin{equation*}
p(\lambda;k)=\lambda^4+3k\lambda^3+3k^2\lambda^2+k^3\lambda+\frac{a}{2}k^3=0
\end{equation*}
and
\begin{eqnarray*}
 p_0(k)&=&\frac{a}{2}k^3,\\
 D_1(k)&=&k^3,\\
 D_2(k)&=&\frac{3}{2}k^4(2k-a),\\
 D_3(k)&=&\frac{1}{2}k^5(16k-9a),\\
 \frac{d D_{3}(k)}{dk}&=&\frac{3}{2}k^4(96k-45a). 
\end{eqnarray*}
Taking $k^*=\frac{9}{16}a$, conditions $(CH1)$ and $(CH2)$ are satisfied and we have an Hopf bifurcation, see Figure \ref{fig:n2macro}. 

\begin{figure}[H]
\center
\subfloat[][$k<k^*$]{
\includegraphics[scale=0.1]{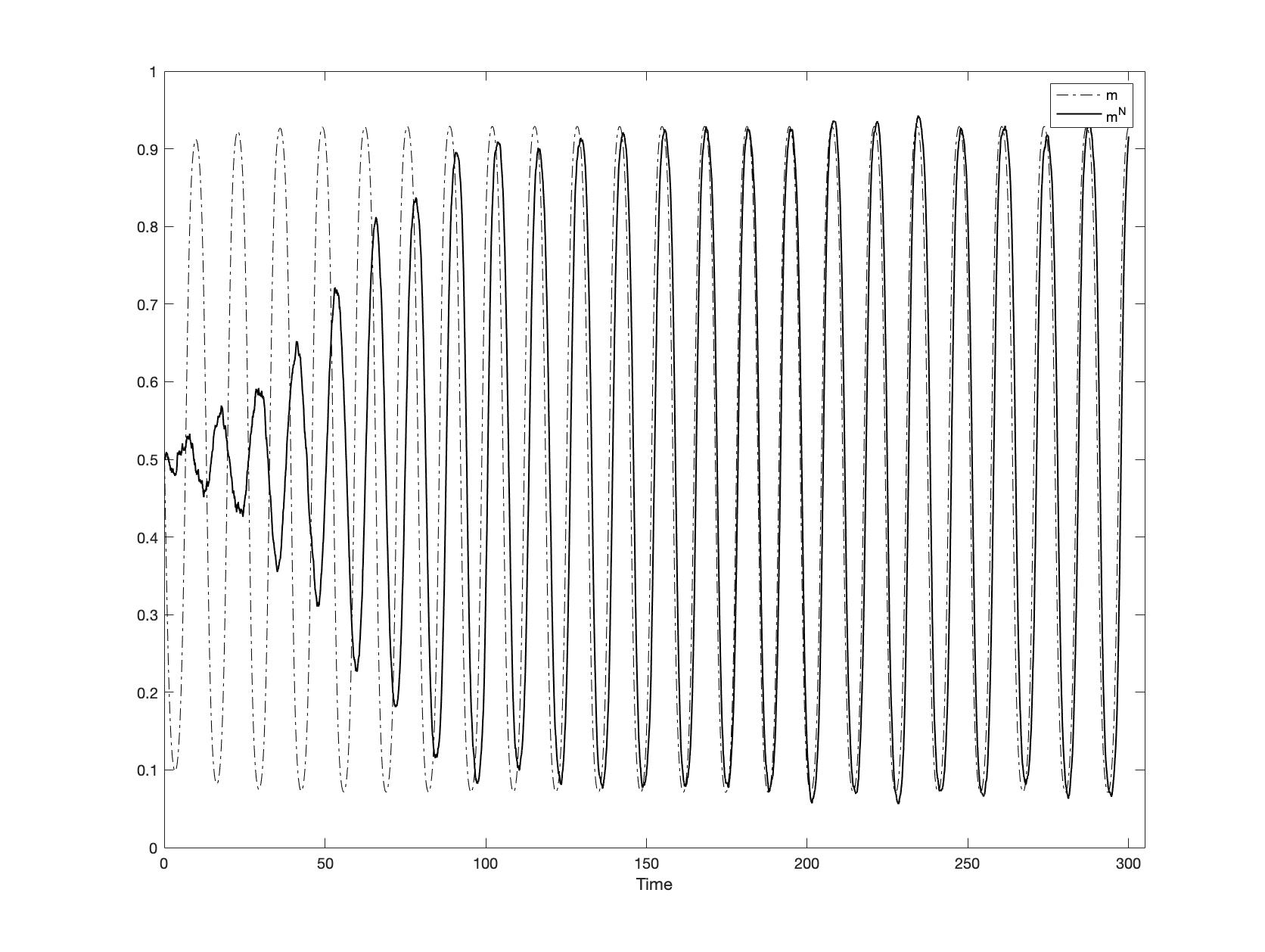}
}
\subfloat[][$k>k^*$]{
\includegraphics[scale=0.15]{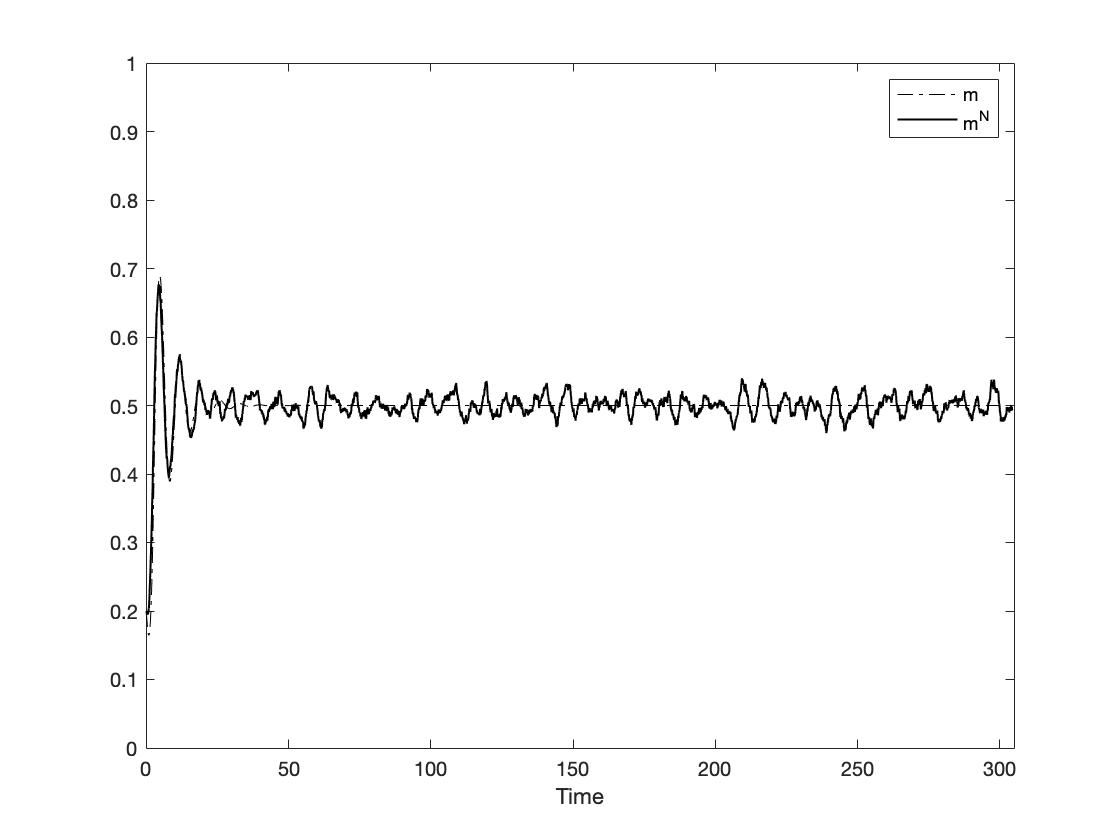}
}\\
\subfloat[][Evolution of the state of the system when $k<k^*$]{
\includegraphics[scale=0.2]{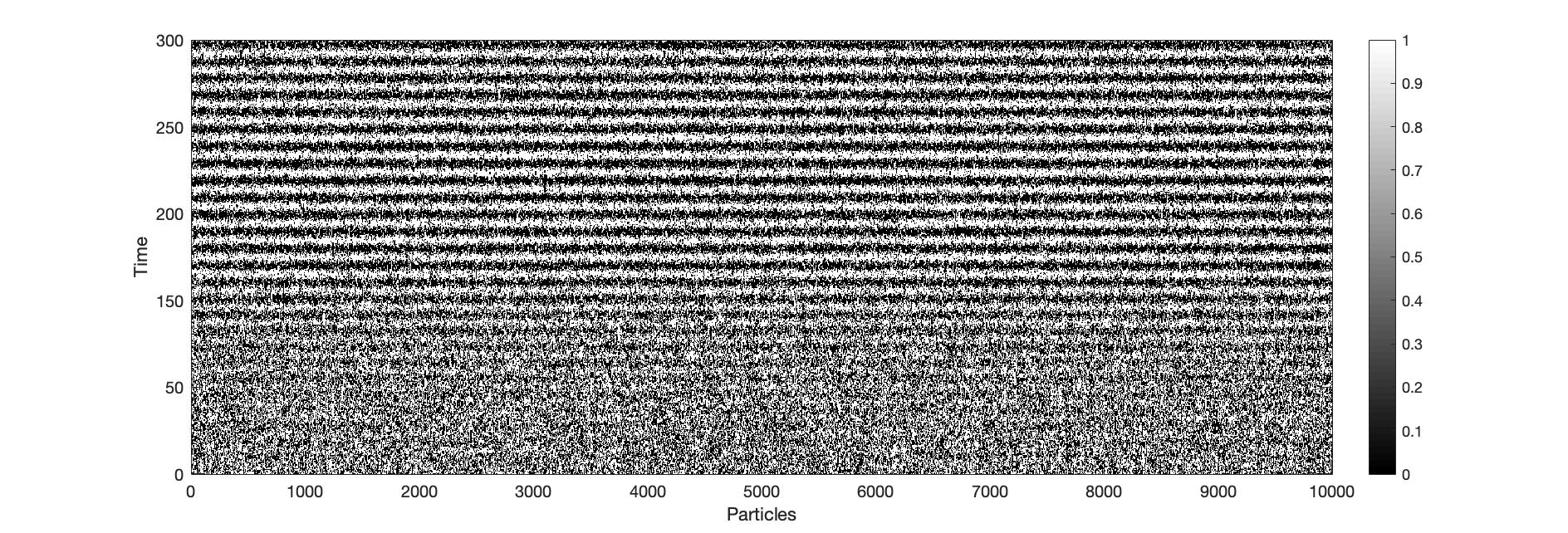}
}
\caption{Simulations for $n=2$: N=10000, $\phi(z)=-2z+3$  }\label{fig:n2macro}
\end{figure}

Now, let fix $n=2$ and consider the case of a general decreasing function $\phi$. 
According to \eqref{eq:charPolyG}, the characteristic polynomial has degree $2n+3$ and we have
\begin{equation*}
p(\lambda;k)= \lambda^7+6k\lambda^6+15k^2\lambda^5+20k^3\lambda^4+k^3(15k-2c)\lambda^3+6k^4(k-c)\lambda^2+k^5(k-6c)\lambda-2ck^6
\end{equation*}
and
\begin{eqnarray*}
 p_0(k)&=&-2ck^6,\\
 D_1(k)&=&k^5(k-6c) ,\\
 D_2(k)&=&2k^9(3k^2-6ck+16c^2),\\
 D_3(k)&=&2k^{12}(35k^3+9ck^2-18c^2k-32c^3),\\
 D_4(k)&=&4k^{15}(224k^3+276ck^3-345c^2k+128c^3), \\
 D_5(k)&=&4k^{16}(2016k^4+2988ck^3-3043c^2k^2+864c^3k-96c^4),\\
 D_6(k)&=&8k^{17}(4096k^4+6048ck^3-6036c^2k^2+1393c^3k-144c^4),\\
 \frac{d D_{6}(k)}{dk} &=& 8k^{16}(86016k^4+120960ck^3-114684c^2k^2+25074c^3k-2448).
\end{eqnarray*}

The term $D_6(k)$ is zero if $k=0$ and if $k$ solves the following equation
\begin{equation*}
4096k^4+6048ck^3-6036c^2k^2+1393c^3k-144c^4=0.
\end{equation*}
Recalling that $c<0$, the previous equation has two real solutions $k_1=r_1c$ and $k_2=r_2 c$, where $r_1<0$ and $r_2>0$ (the explicit expressions of $r_1, r_2$ can be found using a mathematical software and their numerical approximation are $r_1\approx-2.21457$ and $r_2\approx.466319$). Choosing $k^*=k_1$ the conditions $(CH1)$ and $(CH2)$ are satisfied and we have an Hopf's bifurcation. 
We have performed simulations of the macroscopic system and the microscopic one for a particular choice of $\phi$, see Figure \ref{fig:n2tanh}.  
\begin{figure}[!t]
\center
\subfloat[][$k<k^*$]{
\includegraphics[scale=0.15]{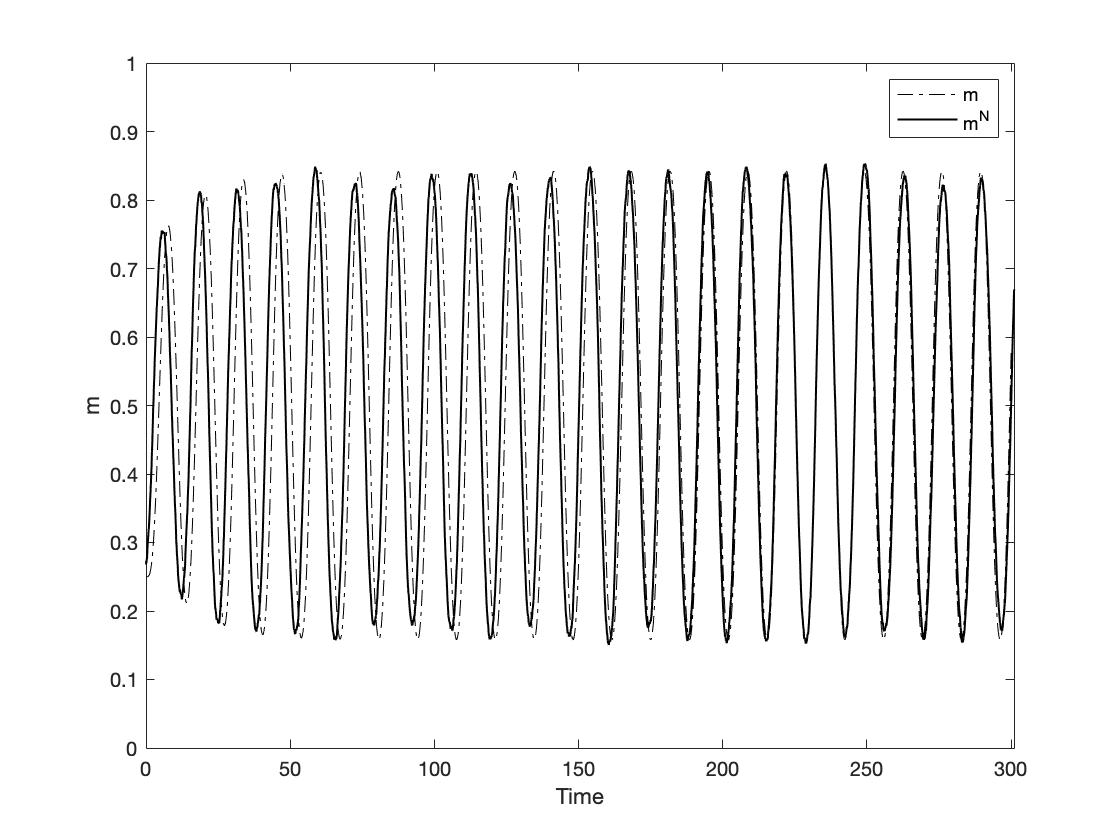}
}
\subfloat[][$k>k^*$]{
\includegraphics[scale=0.15]{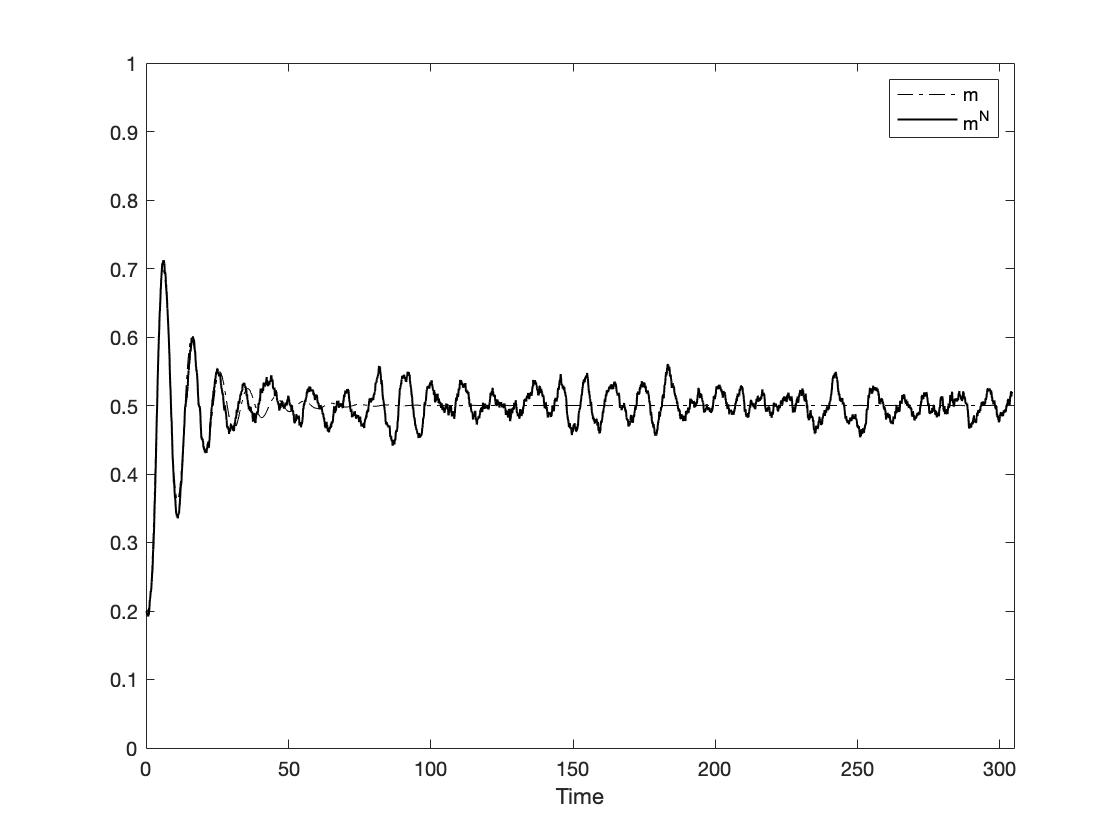}
}\\
\subfloat[][Evolution of the state of the system when $k<k^*$]{
\includegraphics[scale=0.2]{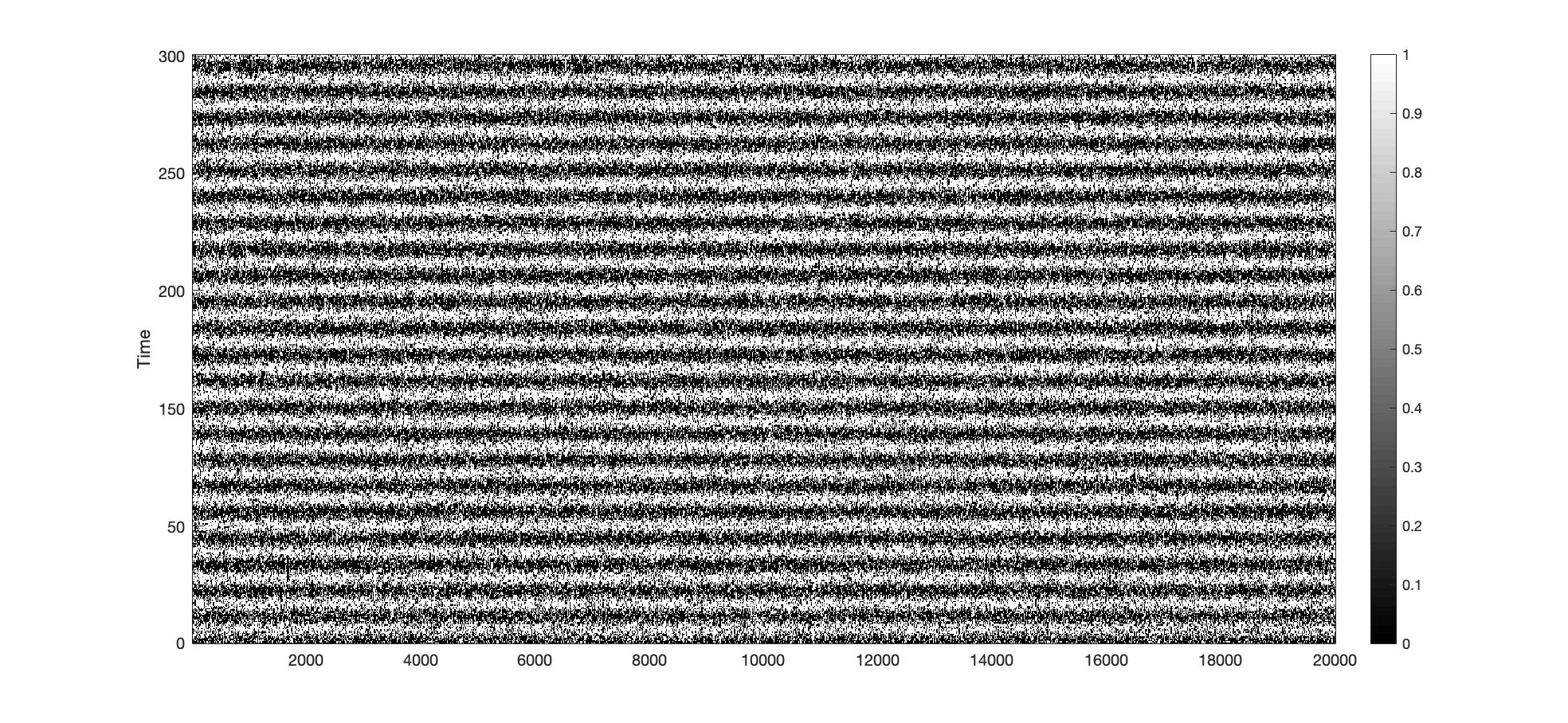}
}
\caption{Simulations for $n=2$: N=20000, $\phi(x)=-\frac{1}{2}\tanh\big(4(x-\frac{1}{2})\big)+3$  }\label{fig:n2tanh}
\end{figure}

We have thus proven the following result
\begin{prop}
Consider the system defined in \eqref{eq:macro}, where $\phi$ is a strictly decreasing function and $n, k$ are the delay parameters, 
 Let $x^*=(\frac{1}{2},\ldots,\frac{1}{2})$. Then, for $n=2$,  a supercritical Hopf bifurcation (with bifurcation parameter $k$) occurs at equilibrium point $x^*$.
\end{prop}

\section{Discussion and conclusions}

We have defined a model of social interactions with binary choices (that can be interpreted as ''actions'', or ''opinions'')  where agents have an intrinsic attitude 
that affects the rate at which they change their choices. Such attitude is expressed in terms of a function $\phi$ of the average choice in the population, that can be interpreted as a ''gain function'' and describes the way agents feel a group pressure coming from their peers. 
We distinguish between populations of conformist (cooperative) or contrarian (competitive) agents, secondly if the function 
$\phi$ is strictly increasing or decreasing, even if more complex situations may be considered, where agents' attitude changes with the pressure's intensity.  
At the same time, agents feel the effect of social influence, i.e., they have a tendency to adjust their choices to the ones of agents with whom they have a direct interaction. 
At a microscopic level, where almost sure convergence to consensus occurs, 
a contrarian attitude of the agents acts as a noise which breaks the order and produces a disordered metastable phase, where agents continuously change their choices and  there is not a prevailing choice in the community. 
The introduction of delay in the  dynamics can be interpreted as a further rise of perturbation that, however, may help create a new ordered phase, where agents coordinate in a periodic behaviour. \\
Further developments of this model could include the presence of mass media, or the subdivision of the community into groups of agents with different gain functions,  
 in order to study phenomena like polarization or fragmentation of choices.\\
Non linear mean field voter models can be considered as a first level attempt in defining non  trivial dynamics in a social network of interacting agents. The model analysed in this paper, while constituting an oversimplified description of human relationships, when observed at a large scale, manages to illustrate some of the typical patterns that are observed in social communities. 
On the other hand, questions and issues coming from the context of social sciences may be a source of inspiration to design challenging mathematical problems.

\appendix

\section{ Explicit formulas in the linear case}\label{appendixA}
We give explicit formulas for the critical parameter $k^*$ at which  
 the system has supercritical Hopf bifurcation in the cases when $\phi$ is linear and $n=3,4,5$. The values of $k^*$ are obtained by checking that the conditions 
 of Theorem \ref{teo:criterionH} re satisfied. 
 We also show simulations of the microscopic model, which confirm what is theoretically expected.\\ 
 
 For $n=3$  according to formula \eqref{eq:charpoly} we obtain
\begin{equation*}
p(\lambda;k)=\lambda^5+4k\lambda^4+6k^2\lambda^3+4k^3\lambda^2+k^4\lambda+\frac{a}{2}k^4.
\end{equation*}
Then
\begin{eqnarray*}
 p_0(k)&=&\frac{a}{2}k^4,\\
 D_1(k)&=&k^4,\\
 D_2(k)&=&k^6(4k-3a),\\
 D_3(k)&=&\frac{5}{2} k^8(8k-7a),\\
 D_4(k)&=&\frac{1}{4} k^8(256k^2-224ak-a^2), \\
 \frac{d D_{4}(k)}{dk} &=& 2k^8(320k-261a). 
\end{eqnarray*}

Taking $k^*=\frac{1}{16}(5\sqrt{2}+7) a$, conditions $(CH1)$ and $(CH2)$ are satisfied and we have an Hopf's bifurcation, see Figure \ref{fig:n3macro}. \\

\begin{figure}[!h]
\center
\subfloat[][$k<k^*$]{
\includegraphics[scale=0.15]{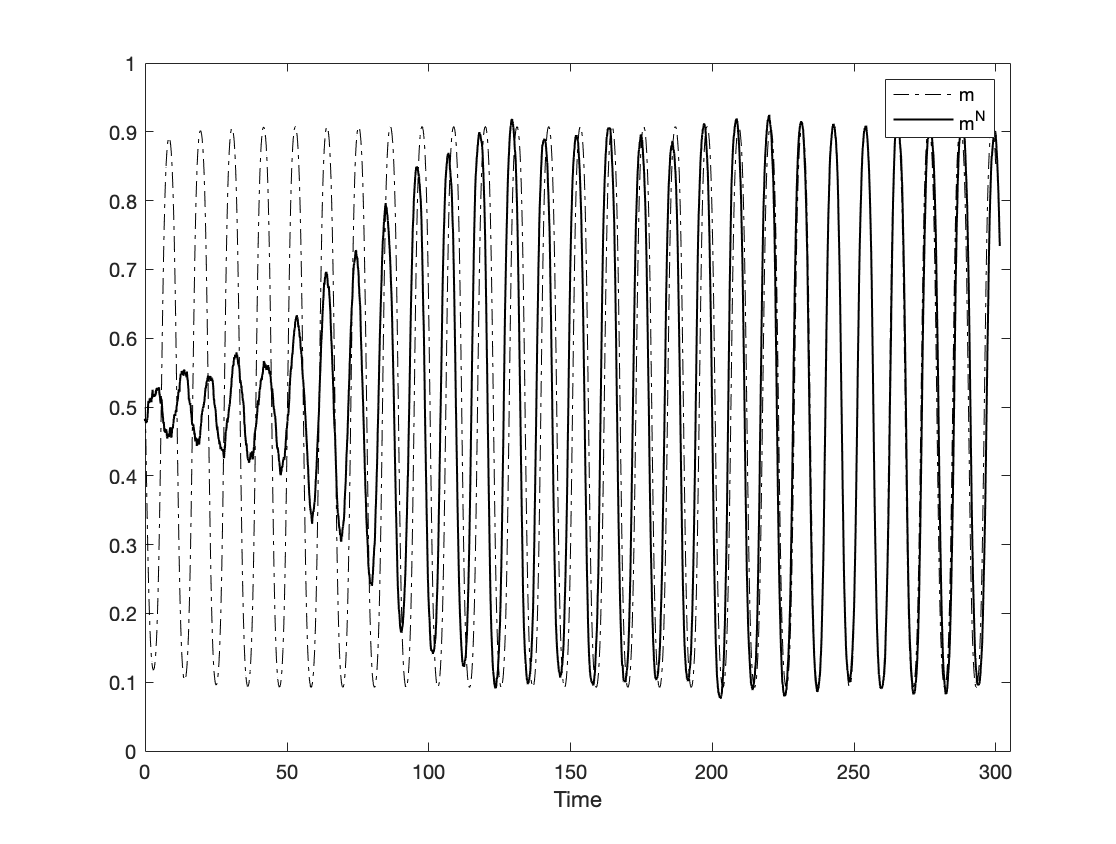}
}
\subfloat[][$k>k^*$]{
\includegraphics[scale=0.15]{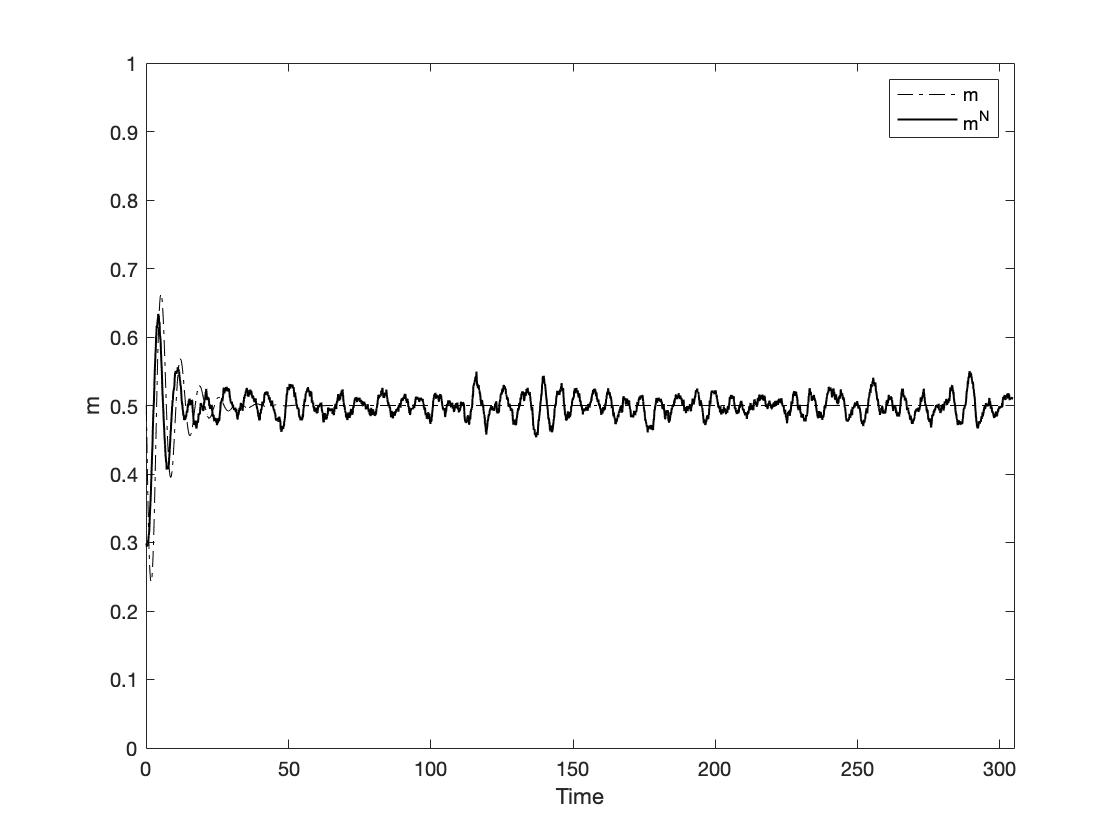}
}\\
\subfloat[][Evolution of the state of the system when $k<k^*$]{
\includegraphics[scale=0.2]{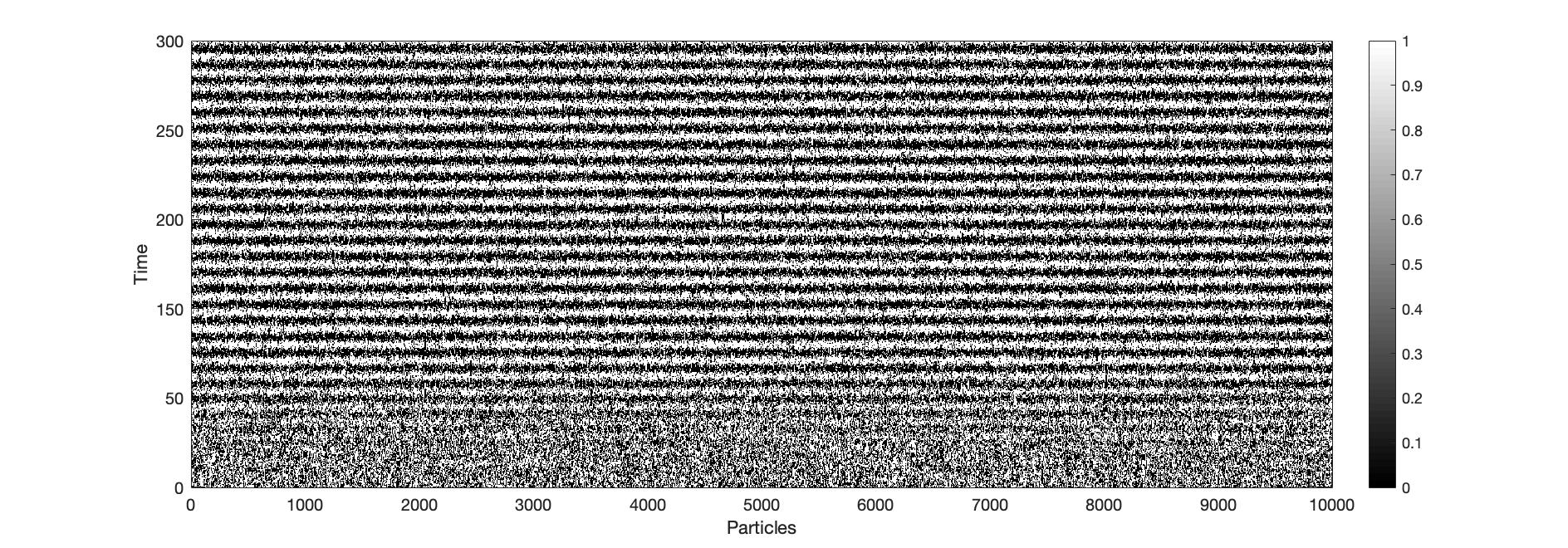}
}
\caption{Simulations for $n=3$: N=10000, $\phi(z)=-2z+3$  }\label{fig:n3macro}
\end{figure}

For $n=4$ we have
\begin{equation*}
p(\lambda;k)=\lambda^6+5 k \lambda^5+10 k^2 \lambda^4+10 k^3 \lambda^3+5 k^4 \lambda^2+k^5 \lambda -2 c k^5.
\end{equation*}
Then
\begin{eqnarray*}
 p_0(k)&=&\frac{a}{2}k^6,\\
 D_1(k)&=&k^5,\\
 D_2(k)&=&5k^8(k-a) ,\\
 D_3(k)&=&\frac{5}{2}k^{11}(16k-19a)\\
 D_4(k)&=&\frac{5}{4}k^{12}(224k^2-264ak-5a^2),\\
 D_5(k)&=&\frac{1}{4}k^{13}( 4096k^2-4800ak-125a^2), \\
 \frac{d D_{5}(k)}{dk} &=&\frac{5}{4} k^{12} (12288 k^2-13440a k-325 a^2) .
\end{eqnarray*}
Taking $k^* = \frac{5}{128} (7\sqrt{5} + 15 )a$ conditions $(CH1)$ and $(CH2)$ are satisfied and we have an Hopf's bifurcation. \\

Finally, for $n=5$ we get
\begin{equation*}
p(\lambda;k)=\lambda^7+6k\lambda^6+15k^2\lambda^5+20 k^3\lambda^4+20 k^3\lambda^4+15 k^4 \lambda^3+6 k^5 \lambda^2+k^6 \lambda-2 c k^6.
\end{equation*}
Then
\begin{eqnarray*}
 p_0(k)&=&\frac{a}{2}k^6,\\
 D_1(k)&=&k^6,\\
 D_2(k)&=&\frac{3}{2} k^{10}(4k-5a) ,\\
 D_3(k)&=&35k^{14} (2k-3a)\\
 D_4(k)&=&\frac{7}{2}k^{16}(256 k^2- 376 a k-15 a^2),\\
 D_5(k)&=&\frac{7}{2} k^{18}   2304 k^2- 3354 a k -209 a^2 ), \\
 D_6(k)&=&\frac{1}{8} k^{18}(4k+a)(16384k^2-24832ak+a^2),\\
 \frac{d D_{6}(k)}{dk} &=&\frac{3}{4} k^{17}  (917504 k^3 - 1269760 a k^2 - 78584 a^2 k +3 a^3   ) .
\end{eqnarray*}
Taking $k^* = \frac{1}{128} (56\sqrt{3} + 97 )a$ conditions $(CH1)$ and $(CH2)$ are satisfied and we have an Hopf's bifurcation, see Figure  \ref{fig:n5macro}.\\

\begin{figure}[!h]
\center
\subfloat[][$k<k^*$]{
\includegraphics[scale=0.15]{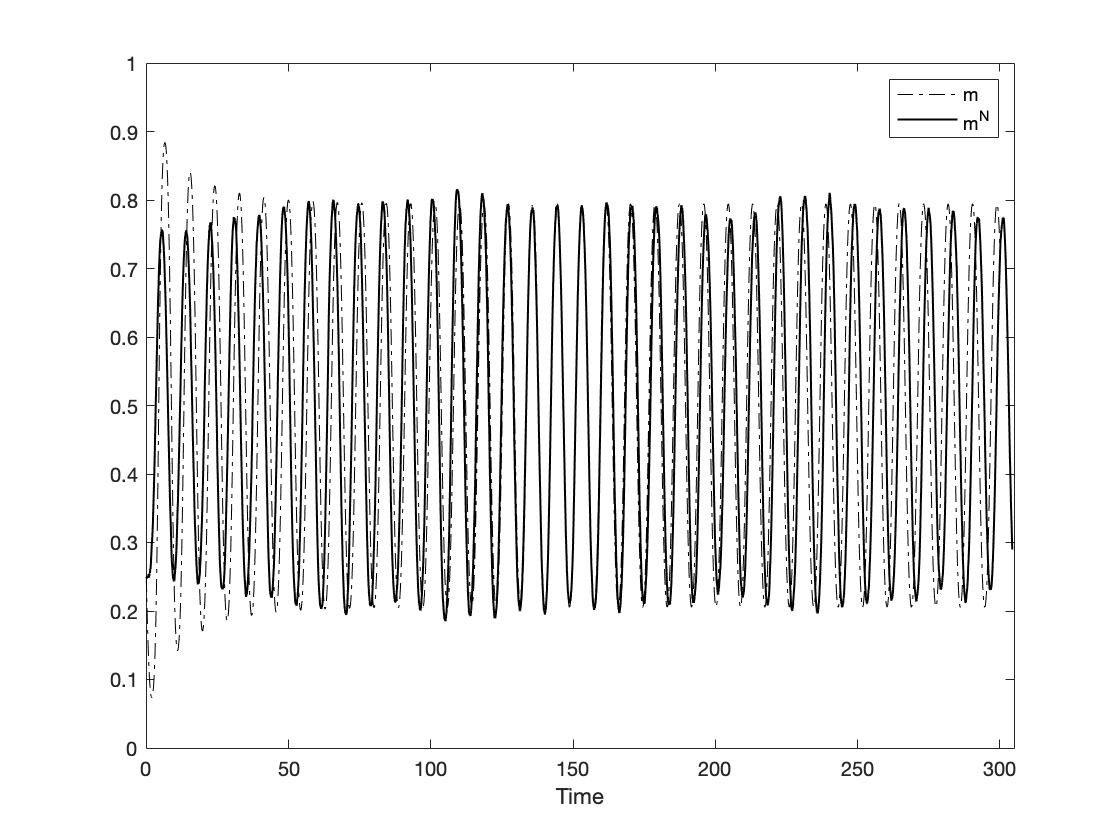}
}
\subfloat[][$k>k^*$]{
\includegraphics[scale=0.15]{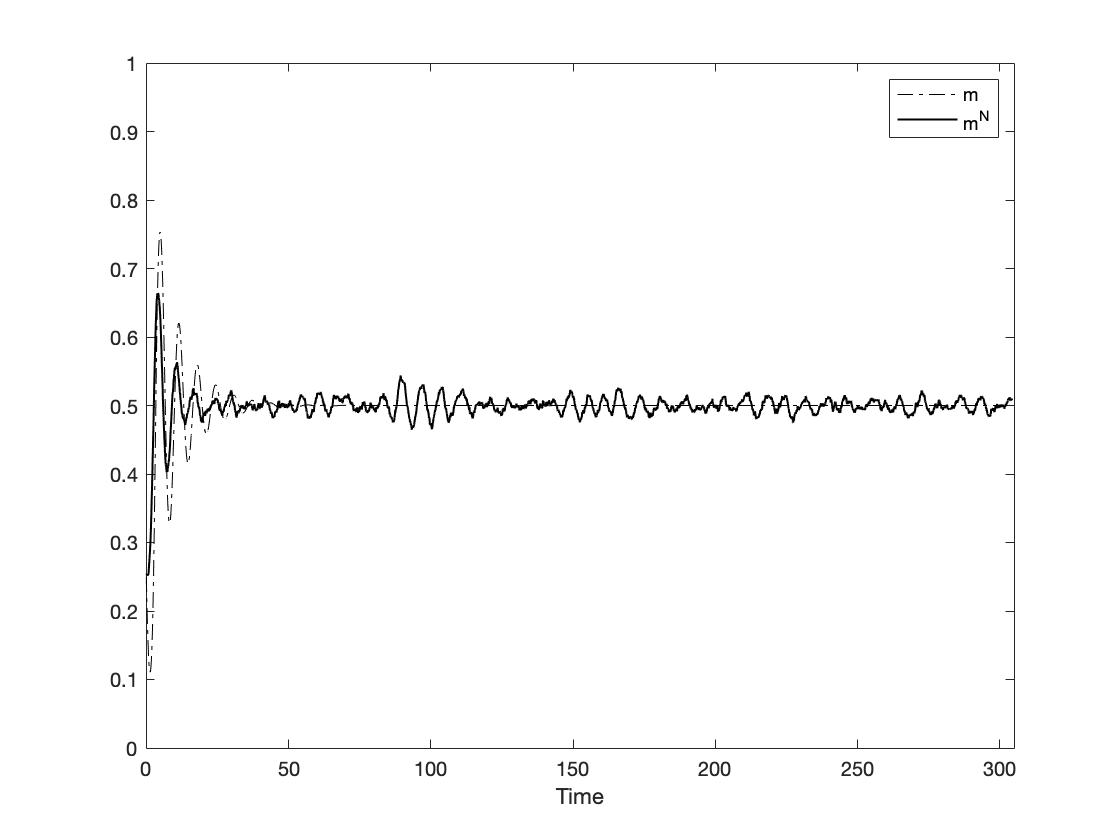}
}\\
\subfloat[][Evolution of the state of the system when $k<k^*$]{
\includegraphics[scale=0.2]{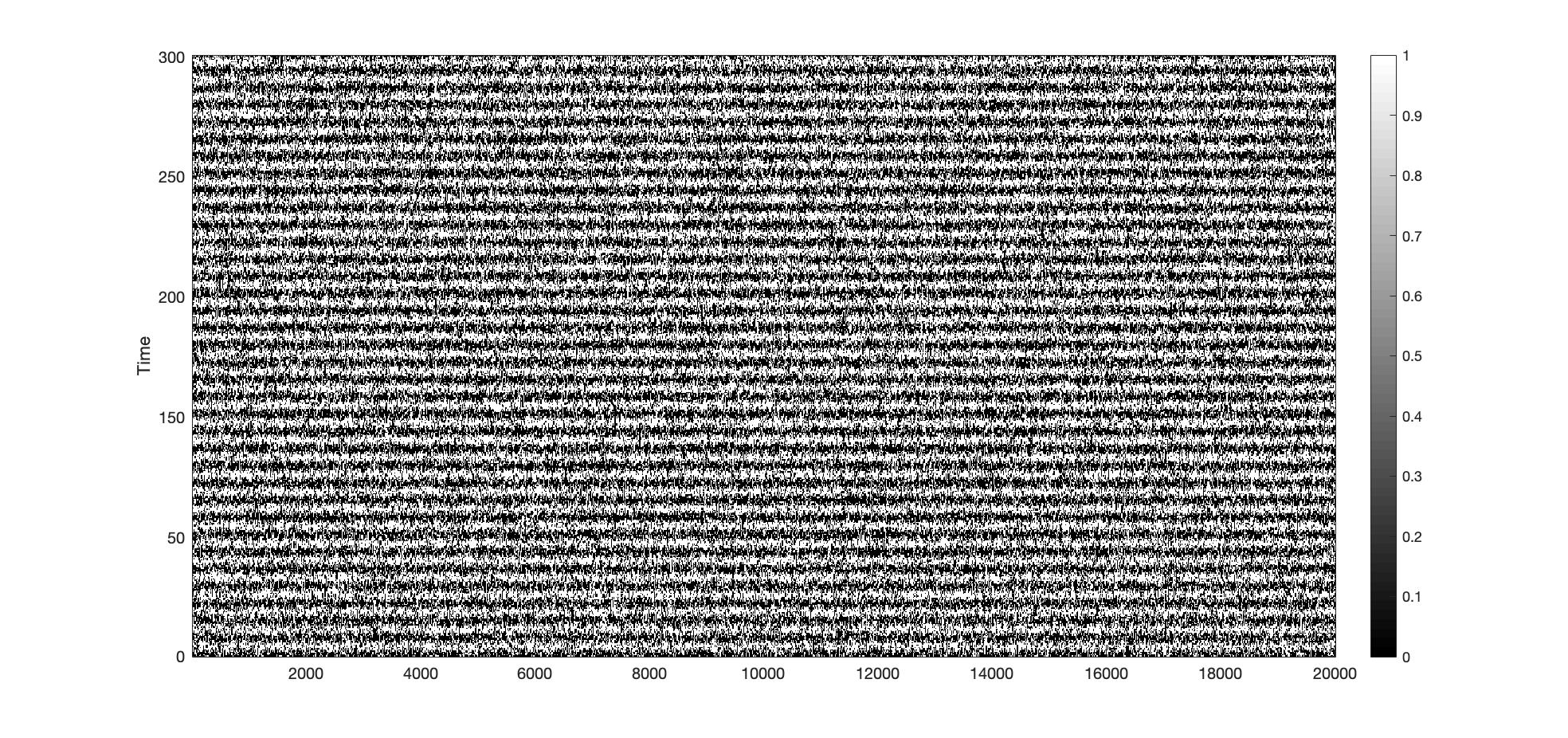}
}
\caption{Simulations for $n=5$: N=20000, $\phi(z)=-2z+3$  }\label{fig:n5macro}
\end{figure}

\section{Auxiliary results}
The general statement of Kurtz's Theorem can be found in  \cite{kurtz70}. Here we write its specialized version to the case of a birth and death process with values in $E_N=\{\frac{i}{N}: 0\leq i\leq N\}$, as done in \cite{fagnani2016proc}.

\begin{theorem}[Kurtz's Theorem]\label{teo:kurtz}
Let $\{\bm{X}(t)\}_{t\geq0}$ be a birth and death process on the state-space $E_N$ with transitions rates, respectively, $r^+(x)=N f^+(x)$ and $r^-(x)=Nf^-(x)$ where $f^+$ and $f^-$ are Lipschitz continuous functions with constants $L^+$, $L^-$. Suppose that $\bm{X}(0) = x_0$ deterministically. Consider the Cauchy problem:
\begin{equation*}
\begin{cases}
\dot{x}(t)=f^+(x)-f^-(x);\\
x(0)=x_0.
\end{cases}
\end{equation*}

Then for any $\epsilon>0$ and $T>0$, for $N $ sufficiently large,
\begin{equation*}
P(\sup_{0\leq t\leq T}|\bm{X}(t)-x(t)|>\epsilon)\leq 4 \exp\left(-NTdg\left(\frac{\epsilon e^{-\bar{L}T}}{4T\bar{f}}\right) \right)
\end{equation*}
where $d=\|f^+-f^-\|_{\infty}$, $g(t)=(1+t)\log(1+t)-t$ and $\bar{L}=\max\{L^+,L^-\}$
\end{theorem}
 
For the proof of the following lemmas, see  \cite{fagnani2016proc}.
\begin{lemma}\label{lem: BD zero}
For $N\geq 1$, let $\{\bm{X}(t)\}_{t\geq0}$ be a birth and death process on $E_N=\{\frac{i}{N}: 0\leq i\leq N\}$ with rates $r^+(x)$ and $r^-(x)$ respectively. Assume $r^+(0)=r^-(0)=0$ and that there exists $\bar{\epsilon}>0$ such that 
\begin{equation*}
r^-(x)\geq(1+\delta)r^+(x),\quad \forall\ x\in E_N\cap(0,2\bar{\epsilon}]\\
\end{equation*}
for some $\delta>0$. Then, letting $C=\ln (1+\delta)$ and $P_x=P(\ \cdot\  | \bm{X}(0)=x)$, for any $x<\bar{\epsilon}$ we have 
\begin{equation}
P_x\big(\sup_{t\geq 0}\bm{X}(t)>2\bar{\epsilon}\big)<\bar{\epsilon} Ne^{-\bar{\epsilon} C N}.
\end{equation}
\end{lemma}

\begin{lemma}\label{lem: BD unmezzo}
For $N\geq 1$, let $\{\bm{X}(t)\}_{t\geq 0}$ be a birth and death process on $E_N=\{\frac{i}{N}: 0\leq i\leq N\}$ with rates $r^+(x)$ and $r^-(x)$ respectively. Let $\mu=\max_{x\in E_N}\big(r^+(x)+r^-(x)\big)$. Assume that there exists $x_0\in (0,1)$ and $\bar{\epsilon}>0$ such that $(x_0-\bar{\epsilon}, x_0+\bar{\epsilon})\subset (0,1)$ and  
\begin{equation*}
r^+(x)\geq(1+\delta)r^-(x)\quad\mbox{if }x\in E_N\cap(x_0-\bar{\epsilon},x_0+\bar{\epsilon})\\
\end{equation*}
for some $\delta>0$ and let $p=\frac{1+\delta}{2+\delta}$. Then, for any $x>x_0$ 
\begin{equation*}
P_{x}\left(\inf_{ t\in[0, e^{\bar{\epsilon} C_p N}]} \bm{X}(t)< x_0-\bar{\epsilon}\right)<\left(9\mu^2 +\frac{1}{9 \mu^2}\right) e^{-\bar{\epsilon} C_p N}
\end{equation*}
for a suitable constants $C_p$ which depends only on $p$. 
\end{lemma}

\paragraph{Acknowledgment:} The authors want to thank Prof. Marco Formentin for a useful discussion about numerical simulations. 

\bibliography{bibliografy}
\bibliographystyle{amsplain}

\end{document}